\documentclass{article}
\usepackage{amsthm}
\usepackage{amsfonts}
\usepackage[reqno]{amsmath}
\usepackage{amssymb,comment}
\usepackage{color,psfrag,graphicx}

\newtheoremstyle{theorem}{}{20pt}{\normalfont}{0pt}{\bfseries}{.}
            {0.8pc}{\thmname{#1}\thmnumber{ #2}\thmnote{ \textup{(#3)}}}
\theoremstyle{theorem}
\newtheorem{thm}{Theorem}[section]
\newtheorem*{thm*}{Theorem}
\newtheorem{remark}[thm]{Remark}
\newtheorem*{remark*}{Remark}
\newtheorem{prop}[thm]{Proposition}
\newtheorem{definition}[thm]{Definition}
\newtheorem{cor}[thm]{Corollary}
\newtheorem*{cor*}{Corollary}
\newtheorem{lemma}[thm]{Lemma}
\newtheorem*{lemma*}{Lemma}

\usepackage[pagebackref]{hyperref}
\usepackage{textcomp}

\numberwithin{equation}{section}

\DeclareMathOperator{\supp}{supp}
\DeclareMathOperator{\dist}{dist}

\newcommand{\R}{\mathbb{R}}
\newcommand{\N}{\mathbb{N}}
\newcommand{\Z}{\mathbb{Z}}
\newcommand{\C}{\mathbb{C}}

\newcommand{\cc}{{\mathcal C}}
\newcommand{\cd}{{\mathcal D}}

\newcommand{\calk}{{\mathcal K}}

\newcommand{\co}{{\mathcal O}}

\newcommand{\bas}{\begin{align*}}
\newcommand{\eas}{\end{align*}}
\newcommand{\ba}{\begin{align}}
\newcommand{\ea}{\end{align}}
\newcommand{\bes}{\begin{equation*}}
\newcommand{\ees}{\end{equation*}}
\newcommand{\be}{\begin{equation}}
\newcommand{\ee}{\end{equation}}

\begin{document}

\title{Relations between Kondratiev spaces and refined localization Triebel-Lizorkin spaces\\[0.3cm]
{\small  \em \textdied \  In memory of Benjamin Scharf (1985--2021) \textdied }
}

\author{Markus Hansen\thanks{
Markus Hansen,  Philipps-Universit\"at Marburg, Hans-Meerwein-Str. 6,  35032 Marburg; 
E-mail: {\sf markus.hansen1@gmx.net},
Phone: +49\,6421\,28\,254868, Fax: +49\,6421\,28\,26945.
Research supported by the European Research Council (ERC) 
under the grant StG306274.},  Benjamin Scharf\thanks{www.benjamin-scharf.de},  and Cornelia Schneider\thanks{Cornelia Schneider,  FAU Erlangen, Department Mathematik,  Cauerstr. 11,  91058 Erlangen; E-mail: {\sf schneider@math.fau.de},  Phone: +49\,9131\,85\,67207,  Fax: +49\,9131\,85\,67207.}
}

\date{}


\maketitle
%
%


\begin{abstract} We investigate the close relation between certain weighted Sobolev spaces (Kondratiev spaces)  and refined localization spaces from  \cite{Tr06,Tr08}.  
In particular,   using a characterization for refined localization spaces from \cite{Scharf},  we considerably improve an embedding from \cite{Hansen1}.    This   embedding is  of special interest in connection with convergence rates for adaptive approximation schemes. 
\end{abstract}

\vspace{3mm}

{\bf Keywords:}
	Regularity for elliptic PDEs,  Kondratiev spaces,  Besov regularity,  Triebel-Lizorkin spaces,  refined localization spaces.
	
\vspace{3mm}

{\bf AMS Subject Classification (2010):} Primary 35B65; Secondary: 41A46, 46E35, 65N30.



\tableofcontents

\allowdisplaybreaks

\section{Introduction}

While often much is known about existence, uniqueness or regularity properties of solutions to boundary value problems of partial differential equations, rarely an analytic expression for solutions can be found. For that reason numerical algorithms for the constructive approximation of the solution up to a prescribed tolerance are needed.  
Numerical studies clearly indicate that modern adaptive algorithms, where the choice of the underlying degrees of freedom is not a priori fixed but depends on the shape of the unknown solution,  have a lot of potential in this context.  Unfortunately,  they  are hard to implement and to analyze,   
and therefore, it is extremely important to theoretically  investigate under which conditions   adaptive methods provably outperform classical schemes in order to justify their use. 
Regarding adaptive approximation,  in this paper we want to contribute  to answer  the following

\textbf{Question:} \textit{Given a boundary value problem for some partial differential equation, what
convergence rate may we expect for an optimal adaptive approximation method?}

To outline the main ideas,  we  consider the model problem
\begin{equation} \label{model-prob}
	-\nabla\cdot(a\nabla u)=f
\end{equation}
on some polygonal domain $D$  in $\R^2$, or a polyhedral domain in $\R^3$. For simplicity, we also restrict ourselves to homogeneous Dirichlet boundary conditions $u|_{\partial D}=0$.  Then the question can be answered in three steps:

\begin{itemize}
	\item
		\textbf{Step 1:} Study the regularity properties of solutions $u$ within an appropriate scale of
		function spaces. Here any family of spaces describing smoothness of functions seems suitable, if
		they allow for an isomorphism $f\mapsto u$. 
	\item
		\textbf{Step 2:} Embed those smoothness spaces into a second scale of function spaces adapted to the
		preferred approximation scheme. The latter spaces are determined via
	\item
		\textbf{Step 3:} Determine concrete function spaces with an optimal embedding into the approximation
		classes associated to the approximation scheme -- optimal in the sense that those should be as large
		as possible while still  maintaining favourable properties.
\end{itemize}

We will essentially be concerned with making substantial progress in view of  Step 2.  But before we   explain our results, let us briefly summarize what is known for Steps 1 and 3 so far.




{\bf Step 3:} 
Given an adaptive algorithm based on an approximation with finite elements or wavelets for the solution spaces of the PDE, the best one can expect is an optimal performance in the sense that it realizes the convergence rate of best $N$-term approximation schemes, which serves as a benchmark in this context. This can be rephrased in terms of approximation classes and one is interested in the 'right' smoothness spaces which embed into these approximation classes.  Thus,  we observe that   Step 3 can be dealt with completely independent of any regularity considerations for PDEs. 

As for the right smoothnes spaces, for many dictionaries, in particular for wavelet bases and frames, it has been shown that the order of convergence that can be achieved depends on the regularity of the object one wants to approximate in the specific so-called 'adaptivity scale' of Besov spaces
\begin{equation}\label{eq:besovscale}
	B^{m+s}_{\tau,q}(D)
		\,,\qquad m = d\,\Big(\frac{1}{\tau}-\frac 12\Big)\,,\quad\tau<2\,,\quad s\geq 0\,.
\end{equation}
We refer e.g. to \cite{BDDP02}, \cite{DNS06}, and \cite{DJP92} for details. More recently, it has also turned out that similar relations hold for finite element approximations, see \cite{GM}. On the other hand, the performance of nonadaptive (uniform) methods is determined by the $L_2$--Sobolev smoothness of the solution, see, e.g., Hackbusch \cite{Hack92} and\cite{DDD} for details. Therefore, the use of adaptivity is justified if the Besov smoothness of the exact solution to an operator equation within this scale \eqref{eq:besovscale} is high enough compared to the classical Sobolev smoothness. These relations are clearly the reason why we are highly interested in regularity estimates in the scale \eqref{eq:besovscale}.

It is nowadays classical knowledge that the Sobolev regularity of the solutions to elliptic problems
depends not only on the properties of the coefficients and the right-hand side, but also on the regularity/roughness of the boundary of the underlying domain. For linear problems, the following relations hold: While for smooth coefficients and smooth boundaries we have $u\in H^{s+2}(D)$ when
$f\in H^s(D)$, it is well-known that this becomes false for more general domains. In particular, if we only assume $D$ to be a Lipschitz domain, then it was shown by Jerison and Kenig \cite{JK} that in general we only have $u\in H^s$ for all $s\leq 3/2$ for the solution of the Poisson equation, even for smooth right-hand side $f$. This behaviour is caused by singularities near the boundary. 

In particular, the $H^{3/2}$-Theorem implies that the optimal rate of convergence for nonadaptive methods of approximation is just $3/{2d}$ (meaning for an approximative solution $u_N$ derived from a nonadaptive method the error $\|u-u_N|L_2(\Omega)\|$ behaves like $\co(N^{-3/{2d}})$) as long as we do not impose further restrictions on $\Omega\subset\R^d$. Similar relations also hold for more specific domains such as domains  of polyhedral or polygonal type, see, e.g., \cite{Gri85, Gri92, Dauge}. However, the norms considered in \eqref{eq:besovscale} are weaker than the Sobolev norm $W^m_2$ and, therefore, there is some hope that the boundary singularities do not influence the smoothness of the solution in the scale \eqref{eq:besovscale} too much (which implies that adaptivity pays off in this case).

{\bf Step 1:}
The result of Jerison and Kenig can also be re-interpreted as a limitation for the standard scale of Sobolev spaces $H^s$ to properly describe the regularity properties of solutions to problems on nonsmooth domains. 
On domains with edges and corners, these nonsmooth parts of the boundary induce singularities for the solution and its derivatives,  which diminish the Sobolev regularity but can be compensated with suitable weight functions.
By means of so-called Kondratiev spaces $\calk^m_{a,p}(D)$, i.e., weighted Sobolev spaces which are introduced via the norm
\[
	\|u|\calk^m_{a,p}(D)\|^p
		=\sum_{|\alpha|\leq m}\int_D|\rho(x)^{|\alpha|-a}\partial^\alpha u(x)|^p\,dx, 
\]
 where the weight $\rho(x)$ measures the distance to the singular set of the boundary $\partial D$ (e.g. the edges and vertices),  it is possible to describe very precisely the behaviour of these singularities.  Moreover, these specific smoothness spaces allow for certain shift  theorems (similar as the ones for smooth domains and Sobolev spaces) in the following sense. Suppose that we are given a second order elliptic differential equation on a polygonal or polyhedral domain. Then, under certain conditions on the coefficients and on the domain, it turns out that if the right-hand side has smoothness $m-1$ in the scale of Kondratiev spaces, then the solution $u$ of the PDE has smoothness $m + 1$.  
Regarding our model problem \eqref{model-prob}, we refer to  
Proposition \ref{prop-regularity} below,  where Step 1 is dealt with.

{\bf Step 2:} In view of the explanations given for Steps 1 and 3, it becomes clear that the missing link   in Step 2 is an embedding between Kondratiev spaces and the specific scale of Besov spaces from  \eqref{eq:besovscale}.  Note that instead of Besov spaces one can also work with the closely linked Triebel-Lizorkin spaces here. 
Forerunners, i.e., regularity estimates in quasi-Banach spaces according to \eqref{eq:besovscale} have  been developed   in 
\cite{Dahlke1, Dahlke2, Dahlke3,  DDD, DD,DS}  (this list is clearly not complete).  Later on, those were reformulated and extended in terms of embeddings of Kondratiev spaces into Besov and Triebel-Lizorkin spaces. 
Recently, in \cite{Hansen1} the embedding
\begin{equation}\label{old-emb}
	\calk^m_{a,p}(D)\cap B^s_{p,\infty}(D)\hookrightarrow B^m_{\tau,\infty}(D)
\end{equation}
with $\tau$ as in \eqref{eq:besovscale} for suitable ranges of parameters was shown, which   was  further sharpened in \cite{Hansen2}. 
The main result of this paper  is a substantial improvement of those embeddings,  invoking as a new tool the 
 so-called refined localization spaces $F^{s,\text{rloc}}_{p,q}(D)$, a modification of Triebel-Lizorkin spaces based on localization procedures.  Since these spaces are actually  smaller than the Triebel-Lizorkin spaces,  this eventually  leads to 
our main  Theorem \ref{thm-embedding-2}  which gives sharp results for the embedding 
\begin{equation}\label{new-emb}
\calk^m_{a,p}(D)\hookrightarrow F^m_{\tau,2}(D) 
\end{equation}
with $\tau$ as in \eqref{eq:besovscale}. 
In particular, we see that the intersection on the left-hand side in \eqref{old-emb} can be avoided.  This is of uttermost importance since the unnecessary intersection with the spaces $ B^s_{p,\infty}(D)$ leads to additional restrictions for the upper bounds of the smoothness $m$ on the right-hand side, which according to Step 3 is directly related with the (best possible) convergence rate of the adaptive algorithms. 
Finally, using \eqref{new-emb} together with the elementary embedding
$F^m_{\tau,2}(D)\hookrightarrow F^m_{\tau,\infty}(D)\hookrightarrow B^m_{\tau,\infty}(D)$,  we see that the task from Step 2 is solved.   
To round up our investigations, we combine   Steps 1 and 2 in Theorem \ref{thm-reg-pde}, where we  apply the embedding \eqref{new-emb} with Proposition \ref{prop-regularity}  in order to obtain a regularity result for boundary value problems for elliptic PDEs.  This can immediately be used in view of Step 3  to derive convergence rates for adaptive wavelet and finite element algorithms.\\




The paper is structured as follows: We first present the definitions and basic properties for Kondratiev and  refined localization spaces in Section \ref{sect-2}. Then, in Section \ref{sect-3} the close relation between
Kondratiev spaces and refined localization spaces is revealed, as for suitable domains $D$ we can identify $F^{m,\text{rloc}}_{p,2}(D)$ with the Kondratiev space $\calk^m_{m,p}(D)$.  Moreover, we present localization properties Kondratiev spaces and deal with their invariance under diffeomorphisms. 
Finally,  in Section \ref{sec-embedding} we prove necessary and sufficient conditions for the desired embedding between general Kondratiev spaces $\calk^m_{a,p}(D)$ and refined localization spaces
$F^{m,\text{rloc}}_{\tau,2}(D)$ and ultimately also $F^m_{\tau,2}(D)$ for $\tau<p$.

\section{Function spaces}
\label{sect-2}

In this section we give the definitions of the different scales of function spaces under consideration.

\subsection{Triebel-Lizorkin spaces}
\label{ssec-Besov}

Besov and Triebel-Lizorkin spaces are nowadays established as being closely related to many approximation schemes, starting with approximating periodic functions by trigonometric polynomials (where the full scale of Besov spaces actually first emerged in the works of Besov 1959/60), free-knot spline approximation (see \cite{devore}), $n$-term wavelet approximation and most recently adaptive finite element schemes \cite{GM}. From this aspect stems the main motivation for us to study embeddings into these scales of function spaces. In this work we shall concentrate on Triebel-Lizorkin spaces only. 
Triebel-Lizorkin spaces can be introduced in a number of (equivalent) ways, including definitions in terms of finite differences or Littlewood-Paley decompositions, or via their wavelet characterization. Here we shall give the Fourier-analytical version in terms of dyadic Littlewood-Paley decompositions.

We start with a function $\varphi_0\in\mathcal{S}(\R^d)$ with $\varphi_0(x)=1$ for $|x|\leq 1$ and $\varphi(x)=0$ for $|x|\geq\frac{3}{2}$. Define $\varphi_1(x)=\varphi_0(2x)-\varphi_0(x)$, and put
$\varphi_j(x)=\varphi_1(2^{-j+1}x)$. Then $\{\varphi_j\}_{j\in\N_0}$ forms a so-called dyadic resolution of unity; in particular, we have $\sum_{j\geq 0}\varphi_j(x)=1$ for every $x\in\R^d$.

Based on such resolutions of unity, we can decompose every tempered distribution $f\in\mathcal{S}'(\R^d)$ into a series of entire analytical functions,
\[
	f=\sum_{j\geq 0}\mathcal{F}^{-1}(\varphi_j\mathcal{F}f)\,,
\]
converging in $\mathcal{S}'$, where $\mathcal{F}$ stands for the Fourier transform. Such Littlewood-Paley decompositions then can be used to define function spaces. In our particular case, for $s\in\R$, $0<p<\infty$ and $0<q\leq\infty$ the Triebel-Lizorkin space $F^s_{p,q}(\R^d)$ is defined as the collection of all distributions $f\in\mathcal{S}'(\R^d)$ such that
\[
	\|f|F^s_{p,q}(\R^d)\|
		=\Biggl\|\biggl(\sum_{j\geq 0}\bigl|\mathcal{F}^{-1}(\varphi_j\mathcal{F}f)(\cdot)\bigr|^q\biggr)^{1/q}
		\Bigg|L_p(\R^d)\Biggr\|
\]
is finite, with a supremum instead of a sum if $q$ is infinite. For more details on the history, equivalent definitions, and properties we refer to the monographs by Triebel \cite{Tr83, Tr92, Tr06}.

Starting from these Fourier-analytical function spaces, the most direct way to introduce spaces on domains (needed  for studying boundary value problems for elliptic PDEs which is our main motivation) is via restriction: We define
\[
	F^s_{p,q}(D):=\bigl\{f\in\cd'(D):\exists\, g\in F^s_{p,q}(\R^d)\,,g\big|_D=f\bigr\}\,,
	\qquad\|f|F^s_{p,q}(D)\|=\inf_{g|_D=f}\|g|F^s_{p,q}(\R^d)\|\,.
\]
Alternative (different or equivalent) versions of this definition can be found, depending on possible additional properties for the distributions $g$ (most often referring to their support). We refer to the monograph \cite{Tr08} for details and references.

A final important aspect of Triebel-Lizorkin spaces are their close relations to many classical function spaces. For our purposes, we particularly mention the identities $F^s_{p,2}(\R^d)=H^s_p(\R^d)$ and $H^m_p(\R^d)=W^m_p(\R^d)$, $1<p<\infty$, $m\in\N$, $s\in\R$, which for Lipschitz domains $D$ transfer directly to the respective scales of function spaces on $D$.

\subsection{Kondratiev spaces}

Another scale of function spaces we are interested in is a type of weighted Sobolev spaces. These spaces $\calk^m_{a,p}(D)$, nowadays often referred to as (Babuska-)Kondratiev spaces, play a central role in the regularity theory for elliptic PDEs on domains with piecewise smooth boundary, particularly polygons (in 2D) and polyhedra (in 3D). The basic idea behind using this scale of spaces, as opposed to the usual scale of Sobolev spaces used in connection with sufficiently regular domains, is to compensate for singularities which are known to emerge at the boundary even for smooth data.  

More precisely, if we only assume the domain $D$ to be Lipschitz, it was shown in \cite{JK} that in general we only have $u\in H^{3/2}(D)$ for the solution of the Poisson equation, even for smooth right-hand side $f$. Similar results on polygons are known since the works of Kondratiev (\cite{Kon67,Kon70,KO83}); we refer to \cite{Gri92} for an overview. In the latter case, it is known that a solution to Poisson's equation can always be decomposed into a regular part $u_R\in H^{s+2}(D)$ (where the right-hand side $f$ belongs to $H^s(D)$) and a singular part $u_S$ which in turn is a linear combination of a finite number of explicitly known singularity functions. These singularity functions are smooth in the interior of the polygon, but exhibit a polynomial blow-up at the vertices.

To compensate for this kind of behaviour while still staying close to ordinary Sobolev spaces, weights are introduced, which finally leads to considering the norm
\begin{equation}\label{K-norm-domains}
	\|u|\calk^m_{a,p}(D)\|^p
		=\sum_{|\alpha|\leq m}\int_D|\rho(x)^{|\alpha|-a}\partial^\alpha u(x)|^p\,dx
\end{equation}
for functions admitting $m$ weak derivatives in $D$. Therein $1<p\leq\infty$ (with the usual modification for $p=\infty$), $a\in\R$, and $\rho:D\longrightarrow [0,1]$ is the smooth distance to the singular set of $D$. This means $\rho$ is a smooth function, and in the vicinity of the singular set it is equivalent to the distance to that set. In 2D this singular set consists exactly of the vertices of the polygon, while in 3D it consists of the vertices and edges of the polyhedra. In case of mixed boundary conditions the singular set further includes points where the boundary conditions change (which can be interpreted as vertices with interior angle $\pi$), and for interface problems points where the interface touches the boundary, respectively in higher dimensions. Note that in general polygonal domains need not to be Lipschitz, in particular, the definition of the Kondratiev spaces and some related regularity results allow for cracks in the domain, which in turn corresponds to vertices with interior angle $2\pi$. In case $p=2$ we simply write $\calk^m_a(D)$.  For more information regarding  Kondratiev spaces and their properties,  we refer to \cite{HS24}.

Within this scale of function spaces, a typical regularity result for boundary value problems for elliptic PDEs can be formulated as follows, see \cite{BMNZ} and the references given there:

\begin{prop}\label{prop-regularity}
	Let $D$ be some bounded polyhedral domain without cracks in $\R^d$, $d=2,3$. Consider the problem
	\begin{equation}\label{eq:PDE}
		-\nabla\bigl(A(x)\cdot\nabla u(x)\bigr)=f\quad\text{in}\quad D\,,\qquad u|_{\partial D}=0\,,
	\end{equation}
	where $A=(a_{i,j})_{i,j=1}^d$ is symmetric and
	\[
		a_{i,j}\in\calk^m_{0,\infty}=\bigl\{v:D\longrightarrow\C:
			\rho^{|\alpha|}\partial^\alpha v\in L_\infty(D)\,,|\alpha|\leq m\bigr\}\,,
		\qquad 1\leq i,j\leq d\,.
	\]
	Let the bilinear form
	\[
		B(v,w)=\int_D\sum_{i,j}a_{i,j}(x)\partial_i v(x)\partial_j w(x)dx
	\]
	satisfy
	\[
		|B(v,w)|\leq R\|v|H^1(D)\|\cdot\|w|H^1(D)\|\qquad\text{and}\qquad r\|v|H^1_0(D)\|^2\leq B(v,v)
	\]
	for all $v,w\in H^1_0(D)$ and  for some constants $0<r\leq R<\infty$. Then there exists some $\overline a>0$ such that for any
	$m\in\N_0$, any $|a|<\overline a$,  and any $f\in\calk^{m-1}_{a-1}(D)$ the problem \eqref{eq:PDE} admits
	a uniquely determined solution $u\in\calk^{m+1}_{a+1}(D)$, and it holds
	\[
		\|u|\calk^{m+1}_{a+1}(D)\|\leq C\,\|f|\calk^{m-1}_{a-1}(D)\|
	\]
	for some constant $C>0$ independent of $f$.
\end{prop}

In the literature there are further results of this type, either treating different boundary conditions, or using slightly different scales of function spaces, see \cite{mazja1,mazja2,NistorMazzucato}.

\paragraph{Kondratiev spaces on $\R^d$}

Apart from the previously introduced Kondratiev spaces on polygonal or polyhedral domains, in this work we shall also consider Kondratiev spaces on the whole of $\R^d$. These are connected to the corresponding  spaces on domains by a result proved in \cite{Hansen1}, the boundedness of the Stein extension operator
$\mathfrak{E}:\calk^m_{a,p}(D)\rightarrow\calk^m_{a,p}(\R^d\setminus S)$, which was subsequently used to derive sharp embedding results into Besov and Triebel-Lizorkin spaces, see \cite{Hansen1, Hansen2}. Though in \cite{Hansen1} only the cases $d=2$ and $d=3$ were considered, the arguments can be extended to the case $d>3$ without essential changes. The norm in the spaces $\calk^m_{a,p}(\R^d\setminus S)$ is the direct analog of \eqref{K-norm-domains}, i.e.
\[
	\|u|\calk^m_{a,p}(\R^d\setminus S)\|^p
		=\sum_{|\alpha|\leq m}\int_{\R^d}|\rho_S(x)^{|\alpha|-a}\partial^\alpha u(x)|^p\,dx\,,
		\quad 1<p<\infty\,,
\]
where $\rho_S$ is again the regularized distance to $S$, but now defined on $\R^d$.

In fact, this formulation immediately allows us to generalize this definition to arbitrary closed sets $S\subset\R^d$, where in the sequel we will always assume $S$ to be a null-set w.r.t. the $d$-dimensional Lebesgue measure. Note that Stein \cite[Section VI.3]{Stein} gave a construction for a regularized distance function (i.e. a smooth function equivalent to the usual distance) for arbitrary closed subsets of $\R^d$. As we shall see later on, the spaces obtained in this way are closely related to the so-called refined localization spaces.  Of particular interest are sets 
$$S=\R^\ell\equiv\R^\ell\times\{0\}^{d-\ell}\subset\R^d\qquad \text{for}\qquad 0\leq\ell<d.$$ 
In subsequent investigations,  the (unbounded) domain $D=\R^d\setminus \mathbb{R}^{\ell}$ with boundary $\partial D=\R^\ell$ will represent a prototypical/model situation. 
We shall use localization arguments to transfer results from these special domains to polyhedral domains.

\subsection{Refined localization spaces}

The final scale of function spaces of interest to us are the refined localization spaces. The original definitions, some basic properties and further references can be found in the monographs \cite{Tr06, Tr08}.

Let $D\subset \mathbb{R}^d$ be an arbitrary domain, i.e., a non-empty open set,  with boundary $\partial D$ and $\overline{D}$ its closure.  
Moreover, let $Q_{j,k}=2^{-j}((0,1)^d+k)$ be the open cube with vertex in $2^{-j}k$ and side length $2^{-j}$, $j\geq 0$, $k\in\Z$.  We denote by $2Q_{j,k}$ the cube concentric with $Q_{j,k}$ with side length $2^{-j+1}$. Then there is a collection of pairwise disjoint cubes $\{Q_{j,k_l}\}_{j\geq 0, l=1,\ldots,N_j}$ such that
\begin{equation}\label{whitney-decomp}
	D=\bigcup_{j\geq 0}\bigcup_{l=1}^{N_j}\overline Q_{j,k_l}\,,\qquad
	\text{dist}(2Q_{j,k_l},\partial D)\sim 2^{-j}\,,\quad j\in\N\,,
\end{equation}
complemented by $\text{dist}(2Q_{0,k_l},\partial D)\geq c>0$. This family of cubes constitutes a so-called Whitney decomposition of the domain $D$. For details we refer to \cite[Theorem VI.1, p. 167]{Stein}. 

\begin{definition}\label{def-rloc}
	Let $\{\varphi_{j,l}\}$ be a resolution of unity of non-negative $C^\infty$-functions w.r.t. the family
	$\{Q_{j,k_l}\}$, i.e.
	\[
		\sum_{j,l}\varphi_{j,l}(x)=1\text{ for all }x\in D\,,\qquad
		|\partial^\alpha\varphi_{j,l}(x)|\leq c_\alpha 2^{j|\alpha|}\,,\quad\alpha\in\N_0^d\,.
	\]
	Moreover, we require $\supp\varphi_{j,l}\subset 2Q_{j,k_l}$. Then we define the refined localization spaces
	$F^{s,\text{rloc}}_{p,q}(D)$ to be the collection of all locally integrable functions $f$ such that
	\[
		\|f|F^{s,\text{rloc}}_{p,q}(D)\|
			=\biggl(\sum_{j=0}^\infty\sum_{l=1}^{N_j}\|\varphi_{j,l}f|F^s_{p,q}(\R^d)\|^p\biggr)^{1/p}<\infty\,,
	\]
	where $0<p<\infty$, $0<q\leq\infty$,  and $s>\sigma_{p,q}:=d\left(\frac{1}{\min(1,p,q)}-1\right)$.
\end{definition}

\begin{remark}\label{remark-rloc}
	(i)	In Triebel's orginal definition, he referred to cubes with center in $2^{-j}k$. In that case, we can no longer find a
	partition of the domain $D$ into such cubes with the above properties, but only a cover. However, this still leads to 
	the same spaces, as can be seen by standard arguments (the main aspect being that the number of cubes overlapping
	at any given point of the domain is still uniformly bounded).
	
	\vspace{3mm}
	
	(ii) A resolution of unity with the required properties can always be found: Start with a bump function
	$\phi\in C^\infty(\R^d)$ for $Q=(0,1)^d$, i.e. $\phi(x)=1$ on $Q$ and $\supp\phi\subset 2Q$. Via dilation and
	translation this yields bump functions $\phi_{j,k}$ for $Q_{j,k}$, with $\psi(x)=\sum_{j,k}\phi_{j,k}(x)>0$ for all 
	$x\in D$. The functions $\varphi_{j,k}=\phi_{j,k}/\psi$ have the required properties.
	
	\vspace{3mm}
	
	(iii) It has been shown in \cite[Proposition 4.20]{Tr06} and \cite[Proposition 3.10]{Tr08} that in the case of
	Lipschitz domains these refined localization spaces coincide with the spaces
	\[
		\widetilde F^s_{p,q}(D):=\{f\in F^s_{p,q}(\R^d):\supp f\subset\overline D\}
	\]
	for all parameters $0<p<\infty$, $0<q\leq\infty$ and $s>\sigma_{p,q}$. For other types of domains this is generally no
	longer true. For example, in \cite[Theorem 3.30]{Scharf} for $D=\R^d\setminus\R^\ell$,  $0\leq l<d$,   and $1\leq p,q<\infty$
	explicit descriptions have been given; e.g., for $0<s<(d-\ell)/p$ we then have
	$F^{s,\text{rloc}}_{p,q}(\R^d\setminus\R^\ell)=F^s_{p,q}(\R^d)$. Below, we will primarily be interested in the case
	$F^{s,\text{rloc}}_{p,q}(\R^d\setminus\R^\ell)$, but with regularity $s>\sigma_{p,q}=d(\frac{1}{\min(1,p,q)}-1)$.
	
	\vspace{3mm}
	
	(iv) It is known that for arbitrary domains $D$ and $0<p<\infty$, $0<q\leq\infty$, $s>\sigma_{p,q}$ it holds
	\begin{equation}\label{eq-refined-weight}
		\|u|F^{s,\text{rloc}}_{p,q}(D)\|\sim\|u|F^s_{p,q}(D)\|+\|\delta(\cdot)^{-s}u|L_p(D)\|\,,
	\end{equation}
	where $\delta(x)=\min\left(\dist(x,\partial D),1\right)$. In particular, we always have
	\[
		\|u|F^s_{p,q}(D)\|\lesssim\|u|F^{s,\text{rloc}}_{p,q}(D)\|\,,\quad\text{i.e.}\quad
		F^{s,\text{rloc}}_{p,q}(D)\hookrightarrow F^s_{p,q}(D)\,.
	\]
	For the proof in the case $q<\infty$ we refer to \cite[Proposition 3.22]{Scharf}. Note that in \eqref{eq-refined-weight}  we
	obviously can replace $\delta$ by the regularized distance $\rho$.
	
	A closer inspection of the proof in \cite{Scharf} reveals that the result remains true also for $q=\infty$: While the
	wavelet systems $\{\psi^j_r:j\in\N_0, r=1,\ldots,N_j\}$ considered in these arguments no
	longer constitute bases for $F^{s,\text{rloc}}_{p,\infty}(D)$, they remain representation systems, i.e. every
	function $f\in F^{s,\text{rloc}}_{p,\infty}(D)$ admits a representation
	\[
		f=\sum_{j,r}c^j_r(f)\psi_r^j
	\]
	for suitable coefficients $c^j_r(f)$, converging in $\mathcal{S}'(\R^d)$ as well as
	$F^{s-\varepsilon,\text{rloc}}_{p,\infty}(D)$ for $\varepsilon>0$, $s-\varepsilon>\sigma_{p,q}$. Moreover, also
	the usual (quasi-)norm-equivalence to some sequence space norm on the coefficient sequence carries over.
	Ultimately, these two facts suffice to prove \eqref{eq-refined-weight}.
\end{remark}

The refined localization spaces share many of the properties of the classical Triebel-Lizorkin spaces. For our purposes we will need two of these properties. At this point we  mention a result about pointwise multipliers; later on in Section \ref{ssec-diffeo} we will investigate the behavior under diffeomorphisms.

\begin{lemma}\label{lemma-pointwise}
	Let $D\subset\R^d$ be an arbitrary domain. Further,  let $s>\sigma_{p,q}$, $0<p<\infty$, $0<q\leq\infty$ and
	$r>s$. Then there exists a positive constant $c>0$ such that
	\[
		\|f\cdot g|F^{s,\text{rloc}}_{p,q}(D)\|
			\leq c\,\|g|\cc^r(\R^d)\|\cdot\|f|F^{s,\text{rloc}}_{p,q}(D)\|\,,
			\qquad f\in F^{s,\text{rloc}}_{p,q}(D)\,,g\in\cc^r(\R^d)\,,
	\]
	where $\cc^r(\R^d)=B^r_{\infty,\infty}(\R^d)$ are the H\"older-Zygmund spaces, and $c$ is independent of $f$,
	$g$ and $D$.
\end{lemma}

\begin{proof}
	The result immediately follows from the definition of the spaces $F^{s,\text{rloc}}_{p,q}(D)$ and the
	corresponding multiplier result for $F^s_{p,q}(\R^d)$ (for which we refer to \cite{Tr83, Tr92}), applied to the
	terms $g\cdot\varphi_{j,l}f$. In particular, the constant $c$ is the one in the corresponding estimate for $F^s_{p,q}(\R^d)$,
	independent of $D$.
\end{proof}

\section{A relation between Kondratiev and refined localization spaces}
\label{sect-3}

In this section we prove two results showing a close relation between Kondratiev spaces and the refined localization spaces. We begin with an identity for a particular choice of parameters, followed by a new localization argument for Kondratiev spaces, which in turn will be the basis for an embedding result in the next section.

\subsection{The spaces $\calk^m_{m,p}(D)$}
\label{ssec-Kmm}

Recalling the identity $F^m_{p,2}(\R^d)=H^m_p(\R^d)=W^m_p(\R^d)$ (which transfers also to Lipschitz domains) and keeping in mind the norm-equivalence \eqref{eq-refined-weight} in the special situation $s=m$,   suggests a possible relation between the spaces $\calk^m_{m,p}(D)$ and $F^{m,\text{rloc}}_{p,2}(D)$.  A formalization of  this relation is based on the next lemma.

\begin{lemma}\label{lemma-rloc-derivative}
	Let $D$ be an arbitrary domain in $\R^d$. Let $0<p,q<\infty$, and let a multiindex $\alpha\in\N^d_0$ with
	$s-|\alpha|>\sigma_{p,q}$ be given. Then for every $f\in F^{s,\text{rloc}}_{p,q}$ its derivative $\partial^\alpha f$
	belongs to $F^{s-|\alpha|,\text{rloc}}_{p,q}(D)$, and we have an estimate
	\begin{equation}\label{eq-derivative}
		\|\partial^\alpha f|F^{s-|\alpha|,\text{rloc}}_{p,q}(D)\|
			\lesssim\|f|F^{s,\text{rloc}}_{p,q}(D)\|\,.
	\end{equation}
\end{lemma}

For a proof, we refer to \cite[Proposition 3.22]{Scharf}.  We can now establish a relation between the Kondratiev and the refined localization spaces as follows. 

\begin{thm}\label{thm-Kmm}
	Let $m\in \N$,  $1<p<\infty$, and let $S\subset\R^d$ be an arbitrary closed set such that $|S|=0$. For the domain
	$D=\R^d\setminus S$ we then have
	\[
		\calk^m_{m,p}(D)=F^{m,\text{rloc}}_{p,2}(D)
	\]
	in the sense of equivalent norms. Moreover,
	\begin{equation}\label{eq-equivalent-norm}
		\|u|\calk^m_{m,p}(D)\|^\#
			=\sum_{|\alpha|=m}\|\partial^\alpha u|L_p(D)\|+\|\rho^{-m}u|L_p(D)\|
	\end{equation}
	defines an equivalent norm on $\calk^m_{m,p}(D)$.
\end{thm}

\begin{proof}
	We first note that for these domains $D$ the weight function $\rho$ in the definition of the Kondratiev spaces
	$\calk^m_{m,p}(D)$ is equivalent to the distance function $\delta$ appearing in \eqref{eq-refined-weight}.
	Moreover, we observe that $S$ being closed already implies
	$\partial D=\overline{D}\setminus D\subset\R^d\setminus D=S$, so in particular also $|\partial D|=0$ (here
	$\partial D$ is treated as a subset of $\R^d$, i.e. also measured w.r.t. the $d$-dimensional Lebesgue measure).
	
	\vspace{3mm}
	
	{\bf Step 1:} The embedding $\calk^m_{m,p}(D)\hookrightarrow F^{m,\text{rloc}}_{p,2}(D)$ now is a direct
	consequence of the identity \eqref{eq-refined-weight}. On the one hand, we obviously have
	$\|\rho^{-m}u|L_p(D)\|\leq\|u|\calk^m_{m,p}(D)\|$, and on the other hand we make use of the observation
	$F^m_{p,2}(D)\equiv F^m_{p,2}(\R^d\setminus S)\cong F^m_{p,2}(\R^d)$. The latter follows from the fact that
	due to $|\partial D|=0$ and $F^m_{p,2}(\R^d)\hookrightarrow F^0_{p,2}(\R^d)=L_p(\R^d)$ the space $F^m_{p,2}(D)$
	consists of regular distributions on $D$, which obviously can be identified with the ones on $\R^d$. The embedding then follows
	from $F^m_{p,2}(D)\cong W^m_p(\R^d)\hookleftarrow\calk^m_{m,p}(D)$.
	
	\vspace{3mm}
		
	{\bf Step 2:} For the reverse embedding, we shall apply Lemma \ref{lemma-rloc-derivative}. The estimate
	\eqref{eq-derivative} together with \eqref{eq-refined-weight} particularly implies
	\[
		\|\rho^{|\alpha|-m}\partial^\alpha u|L_p(D)\|\lesssim \|u|F^{m,\text{rloc}}_{p,2}(D)\|
	\]
	for all functions $u\in F^{m,\text{rloc}}_{p,2}(D)$ and all multiindices with $|\alpha|<m$. Moreover, we once
	more use the observation that we can identify $F^m_{p,2}(D)$ with $W^m_p(\R^d)$, which in view of
	\eqref{eq-refined-weight} now yields
	\[
		\|\partial^\alpha u|L_p(D)\|
			\leq\|u|W^m_p(\R^d)\|\lesssim\|u|F^m_{p,2}(D)\|
			\lesssim\|u|F^{m,\text{rloc}}_{p,2}(D)\|\,,\qquad
			|\alpha|=m\,.
	\]
	Altogether this proves $F^{m,\text{rloc}}_{p,2}(D)\hookrightarrow\calk^m_{m,p}(D)$.
	
	\vspace{3mm}
		
	{\bf Step 3:} The norm equivalence for $\calk^m_{m,p}(D)$ now is a consequence of the corresponding
	norm equivalence for $W^m_p(\R^d)$: Recall $F^m_{p,2}(D)=W^m_p(\R^d)$, and
	\[
		\|u|W^m_p(\R^d)\|\sim\|u|L_p(\R^d)\|+\sum_{|\alpha|=m}\|\partial^\alpha u|L_p(\R^d)\|
	\]
	together with $\|u|L_p(D)\|\leq\|\rho^{-m}u|L_p(D)\|$ due to $0<\rho(x)\leq 1$ for all $x\in D$. This yields
	$\|u|\calk^m_{m,p}(D)\|\sim\|u|F^{m,\text{rloc}}_{p,2}(D)\|\lesssim\|u|\calk^m_{m,p}(D)\|^\#$. The reverse
	inequality is obvious.
\end{proof}

\begin{remark}
	The result fails for $F^{m,\text{rloc}}_{p,q}(D)$ when  $q\neq 2$: While the arguments in the first step remain true
	for $q\geq 2$ and those in Step 2 for $q\leq 2$, we indeed have equivalence for $q=2$ only.
\end{remark}

\begin{cor}\label{cor-Kmm}
	The results of Theorem \ref{thm-Kmm} notably apply to domains $\R^d\setminus S$, where $S$ is the singular set
	of some polytope in $D\subset\R^d$. In particular, we have
	\[
		\calk^m_{m,p}(\R^d\setminus S)=F^{m,\text{rloc}}_{p,2}(\R^d\setminus S)
	\]
	as well as an embedding
	\begin{equation}\label{embedding-Kmm}
		\calk^m_{m,p}(D)\hookrightarrow F^m_{p,2}(D)\,.
	\end{equation}
\end{cor}

\begin{proof}
	Though the identity for the domain $\R^d\setminus S$ is fully covered by Theorem \ref{thm-Kmm}, for future
	reference we include a proof reducing this situation to the particular case of domains $\R^d\setminus\R^\ell$.
	
	\vspace{3mm}
		
	{\bf Step 1:} The mentioned localization procedure is similar to the one used in \cite[Appendix A]{Hansen1}. The
	idea is to consider a suitable covering of the singular set by (finitely) many open sets
	$U_1,\ldots,U_N\subset\R^d$. This cover of $S$ is chosen in such a way that every set $U_i$ can
	be identified (after rotation and translation) by a corresponding neighborhood
	$\widetilde U_i\subset\R^d\setminus\R^\ell$ for some $\ell$, i.e., the distance function $\rho_S$ on $U_i$
	coincides with $\rho_{\R^\ell}$ on $\widetilde U_i$. This cover of $S$ then is to be extended with additional
	finitely many open sets $U_{N+1},\ldots,U_M$ to an open cover of $\overline{D}$. On these sets
	$U_{N+1},\ldots,U_M$ the distance function $\rho_S$ shall be bounded from below. Additionally, we require a
	resolution of unity $\{\varphi_j\}_{j=1,\ldots,M}$ w.r.t. $\{U_j\}_{j=1,\ldots,M}$. Accordingly, we decompose
	$f=\sum_{j=1}^M\varphi_j f$, and find by the (quasi-)triangle inequality and Lemma \ref{lemma-pointwise}
	\[
		\|f|F^{m,\text{rloc}}_{p,2}(\R^d\setminus S)\|
			\sim\sum_{j=1}^M\|\varphi_j f|F^{m,\text{rloc}}_{p,2}(\R^d\setminus S)\|\,,
			\qquad f\in F^{m,\text{rloc}}_{p,2}(\R^d\setminus S)\,.
	\]
	By choice of the cover $\{U_j\}_{j=1,\dots,M}$, the terms
	$\|\varphi_j f|F^{m,\text{rloc}}_{p,2}(\R^d\setminus S)\|$ after translation and rotation correspond to Theorem
	\ref{thm-Kmm} for $D=\R^d\setminus\R^\ell$. In view of \eqref{eq-refined-weight} the quasi-norm in
	$F^{m,\text{rloc}}_{p,2}(\R^d\setminus S)$ is invariant under translation and rotation as distances are preserved (note that 
	  this also follows from the considerations for general diffeomorphisms in Section \ref{ssec-diffeo}).
		
	In particular: If $D\subset\R^2$ is a polygon (or a Lipschitz domain with polygonal structure), then
	$S$ consists of finitely many points, which trivially can be covered by $N=\# S$ many, pairwise
	disjoint open sets $U_i$. In case of a polyhedral domain $D\subset\R^3$, the situation is a little more diverse.	
	The cover of $S$ then consists of three types of open sets: The first one covering the interior of exactly one edge
	each, but staying away from all vertices. which clearly corresponds to the case $D=\R^3\setminus\R$. To describe
	the other two types, let $A\in S$ be a vertex, and $\Gamma_1,\ldots,\Gamma_n$ edges with endpoint in $A$. Then
	for every $j$ we can find a cone $C_{A,\Gamma_j}$ with vertex in $A$ and axis $\Gamma_j$ with sufficiently
	small height and opening angle, so that no two such cones intersect. Clearly, in any such cone the distance to $S$ is
	exactly the distance to the axis of the cone (the intersection with $S$ is just the edge $\Gamma_j$). Finally, let
	$\widetilde B_A$ be a ball around $A$ with sufficiently small radius, and denote by
	$\widetilde C_{A,\Gamma_j}$ a cone with half the opening angle of $C_{A,\Gamma_j}$. As the last type of
	neighborhoods we define $B_A$ to be the interior of
	$\widetilde B_A\setminus\bigcup_j\widetilde C_{A,\Gamma_j}$. Then on $B_A$ the distance to $S$ is
	equivalent to the distance to $A$.
	
	\vspace{3mm}
	
	{\bf Step 2:} The embedding \eqref{embedding-Kmm} now is an immediate consequence of the identity for
	$\R^d\setminus S$ and the boundedness of Stein's extension operator $\mathfrak{E}$ on $\calk^m_{m,p}$,
	\[
		\|u|F^m_{p,2}(D)\|
			\leq\|\mathfrak{E}u|F^m_{p,2}(\R^d\setminus S)\|
			\lesssim\|\mathfrak{E}u|F^{m,\text{rloc}}_{p,2}(\R^d\setminus S)\|
			\sim\|\mathfrak{E}u|\calk^m_{m,p}(\R^d\setminus S)\|
			\lesssim\|u|\calk^m_{m,p}(D)\|
	\]
	for all $u\in\calk^m_{m,p}(D)$.
\end{proof}

Unfortunately, the equivalence of the norm \eqref{eq-equivalent-norm} does not immediately extend to the spaces $\calk^m_{a,p}(D)$ for parameters $a\neq m$. However, with the help of the following lemma, we can give a partial analogue. The lemma is of interest on its own, as it relates Kondratiev spaces for different weight parameters $a$.

\begin{lemma}\label{lemma-Kma}
	Let $1<p<\infty$, $a\in\R$,  and $m\in\N$. Define the multiplication operator $T: u\mapsto\rho^{m-a}u$. Then this
	operator is an isomorphism $T:\calk^m_{a,p}(D)\rightarrow\calk^m_{m,p}(D)$, where $D$ is either a polytope or
	of the form $\R^d\setminus S$ as before. Its inverse is the mapping $u\mapsto\rho^{a-m}u$. Consequently, the
	mapping $u\mapsto\rho^{a'-a}u$ is an isomorphism from $\calk^m_{a,p}(D)$ onto $\calk^m_{a',p}(D)$ for every
	$a'\in\R$.
\end{lemma}

\begin{proof}
	Though the result is known, we include its (simple) proof for the sake of completeness.
	Since the invertibility of $T$ is obvious, it remains to prove its boundedness. The proof is based on estimates for
	derivatives of the regularized distance derived by Stein \cite[Theorem VI.2]{Stein}: For every multiindex $\alpha$
	it holds
	\[
		|\partial^\alpha\rho(x)|\leq B_\alpha\bigl(\delta(x)\bigr)^{1-|\alpha|}
			\leq A_\alpha\bigl(\rho(x)\bigr)^{1-|\alpha|}\,,
	\]
	where $\delta(x)=\text{dist}(x,S)$, with constants independent of $S$. From this, corresponding estimates can be
	obtained for powers of $\rho$,
	\[
		\bigl|\partial^\alpha\bigl(\rho(x)^\gamma\bigr)\bigr|
			\leq A_{\alpha,\gamma}\bigl(\rho(x)\bigr)^{\gamma-|\alpha|}\,.
	\]	
	For every term $\partial^\alpha(\rho^{m-a}u)$,  $|\alpha|\leq m$,   we then can argue using Leibniz rule
	\begin{align*}
		\biggl(
			&\int_D\bigl|\rho^{|\alpha|-m}\partial^\alpha(\rho^{m-a}u)\bigr|^p dx\biggr)^{1/p}\\
			&\leq\sum_{\beta\leq\alpha}\binom{\alpha}{\beta}A_{\alpha-\beta,m-a}\biggl(\int_D
				\bigl|\rho^{|\alpha|-m}\rho^{m-a-|\alpha-\beta|}\partial^\beta u\bigr|^p dx\biggr)^{1/p}\\
			&=\sum_{\beta\leq\alpha}\binom{\alpha}{\beta}A_{\alpha-\beta,m-a}\biggl(\int_D
				\bigl|\rho^{|\beta|-a}\partial^\beta u\bigr|^p dx\biggr)^{1/p}
				\lesssim\|u|\calk^m_{a,p}(D)\|\,,
	\end{align*}
	observe $|\alpha-\beta|=|\alpha|-|\beta|$ here.
\end{proof}

Combining this lemma with Theorem \ref{thm-Kmm} we immediately obtain: 

\begin{cor}\label{cor-equiv-norm}
	Let $1<p<\infty$, $a\in\R$,  and $m\in\N$.  The mapping
	\begin{equation*}
		\|u|\calk^m_{a,p}(D)\|^\#
			=\sum_{|\alpha|=m}\|\partial^\alpha(\rho^{m-a}u)|L_p(D)\|+\|\rho^{-a}u|L_p(D)\|
	\end{equation*}
	defines an equivalent norm on $\calk^m_{a,p}(D)$.
\end{cor}


\subsection{Localization of Kondratiev spaces}
\label{ssec-localization}

Additionally to the localization procedure presented when re-proving Corollary \ref{cor-Kmm}, where we did show how to  reduce problems for Kondratiev spaces defined on  general domains of polygonal or polyhedral structure to the standard situation of  spaces on domains $D=\R^d\setminus\R^\ell$, we now present two further important localization results. 


The following lemma was the original incentive to investigate relations between Kondratiev spaces and localized versions of Triebel-Lizorkin spaces (recall that the usual Fourier-analytic definition of Triebel-Lizorkin spaces is based on dyadic decompositions of the Fourier transform). It can be found in \cite[Lemma 1.2.1]{mazja3} (originally formulated for a scale $V^{l,\beta}_p(K)$, which is related to our definition via
$V^{l,\beta}_p(K)=\calk^l_{l-\beta,p}(K)$).

\begin{lemma}
	Let $K\subset\R^d$ be a  cone centered at $0$, and let $(\varphi_j)_{j\in\Z}$ be a smooth dyadic decomposition of
	unity on $K$, i.e. $\varphi_j$ is a nonnegative $C^\infty$-function on $\overline K$ with
	\[
		\supp\varphi_j\subset K_j=\{x\in\overline K:2^{-j-1}\leq |x|\leq 2^{-j+1}\}\,,
		\qquad |\partial^\alpha\varphi_j|\leq c_\alpha 2^{j|\alpha|}\quad\forall\alpha\in\N^d_0\,,
	\]
	and $\sum_{j\in\Z}\varphi_j(x)=1$ for all $x\in K$. Then for all $m\in\N_0$, $1<p<\infty$,  and $a\in\R$ it holds
	\[
		\|u|\calk^m_{a,p}(K)\|^p
			\sim\sum_{j\in\Z}\|\varphi_j u|\calk^m_{a,p}(K)\|^p\,,\qquad u\in\calk^m_{a,p}(K)\,.
	\]
\end{lemma}

For Kondratiev spaces these localization procedures show two main advantages: The subdomains are chosen in such a way that the distance functions are essentially constant on them, i.e. for $j\geq j_0$ we have $\rho(x)\sim 2^{-j}$ on $K_j$. Moreover, simple dyadic scaling $x\mapsto 2^j x$ maps all subdomains $K_j$ onto one reference domain. The latter then allows to transfer results about the classical (unweighted) Sobolev spaces to Kondratiev spaces. In \cite{mazja3} these arguments  were particularly used to derive Sobolev-type embeddings for Kondratiev spaces.\\

The following theorem now deals with an  analogous  localization of  functions from Kondratiev spaces w.r.t. Whitney decompositions of the domain, as already encountered in the definition of the refined localization spaces.


\begin{thm}\label{thm-localize}
	Let $\{Q_{j,k_l}\}_{j\geq 0, l=1,\ldots,N_j}$ be a Whitney decomposition of $D=\R^d\setminus S$ as in
	\eqref{whitney-decomp} together with a corresponding decomposition of unity $\{\varphi_{j,l}\}$.  Then for all $m\in\N_0$, $1<p<\infty$,  and $a\in\R$ it holds
	\[
		\|u|\calk^m_{a,p}(D)\|^p
			\sim\sum_{j,l}\|\varphi_{j,l}u|\calk^m_{a,p}(D)\|^p
	\]
	for all $u\in\calk^m_{a,p}(D)$.
\end{thm}

\begin{proof}
	{\bf Step 1:} A crucial observation about Whitney decompositions and related resolutions of unity is the fact that
	the supports of the functions $\varphi_{j,l}$ have uniformly bounded overlap, i.e. at every $x\in D$ only a
	uniformly bounded number of values $\varphi_{j,l}(x)$ is non-zero. To show this we argue that every point
	$x\in D$ belongs to a uniformly bounded number of cubes $2Q_{j,k_l}$.

	(i) This clearly is true for any family $\{\gamma Q_{j,k}\}_{k\in\Z^d}$, $\gamma\geq 1$ of expanded versions of
	dyadic cubes (here $\gamma Q_{j,k}$ is the cube concentric with $Q_{j,k}$ and sidelength $\gamma 2^{-j}$) for
	fixed $j$. Here the number of overlapping cubes depends only on the dimension and the expansion factor
	$\gamma$. Moreover, this transfers to any subfamily, particularly to $\{2Q_{j,k_l}\}_{l=1,\ldots,N_j}$.

	(ii) Similarly, any point $x\in D$ can only belong to a uniformly bounded number of layers
	$D_j=\bigcup_{l=1}^{N_j}Q_{j,k_l}$. This in turn follows from the conditions
	$c_1 2^{-j}\leq\text{dist}(2Q_{j,k_l},\partial D)\leq 2^{-j}$ for $j>0$ and
	$\text{dist}(2Q_{0,k_l},\partial D)\geq c>0$, which clearly remain true (though with different constants) for
	$Q_{j,k_l}$. Moreover, the number of overlapping layers only depends on the (ratio of the) constants $c,c_1,c_2$.

	Both observations (i) and (ii) combined yield the desired fact.
	
	\vspace{3mm}
	
	{\bf Step 2:} Now let $v\in L_p(D)$. Moreover, denote by
	$I_{j,l}=\{(j',l'):\supp\varphi_{j',l'}\cap Q_{j,k_l}\neq\emptyset\}$. Then, by similar arguments as in Step 1, we
	see $\# I_{j,l}$ is bounded uniformly in $l$ and $j$. Together with the observation from Step 1, we can argue
	\begin{align*}
		\int_D|v(x)|^p dx
			&=\sum_{j,l}\int_{Q_{j,l}}|v(x)|^p dx
				=\sum_{j,l}\int_{Q_{j,l}}\biggl|\sum_{(j',l')\in I_{j,l}}\varphi_{j',l'}(x)v(x)\biggr|^p dx\\
			&\lesssim\sum_{j,l}\sum_{(j',l')\in I_{j,l}}\int_{Q_{j,l}}|\varphi_{j',l'}(x)v(x)|^p dx
				\leq\sum_{j,l}\sum_{(j',l')\in I_{j,l}}\int_{Q_{j,l}}|v(x)|^p dx\\
			&\lesssim\sum_{j,l}\int_{Q_{j,l}}|v(x)|^p dx
				=\int_D|v(x)|^p dx\,.
	\end{align*}
	Consequently, we have
	\[
		\int_D|v(x)|^p dx
			\sim\sum_{j,l}\sum_{(j',l')\in I_{j,l}}\int_{Q_{j,l}}|\varphi_{j',l'}(x)v(x)|^p dx
			\sim\sum_{j',l'}\int_{D}|\varphi_{j',l'}(x)v(x)|^p dx\,;
	\]
	note that we can extend the sum over $I_{j,l}$ in the middle term to all pairs $(j',l')$ without change.
	
	\vspace{3mm}
	
	{\bf Step 3:} Applying the last identity with $v(x)=\rho(x)^{|\alpha|-a}\partial^\alpha u$ for
	$u\in\calk^m_{a,p}(D)$, it remains to show
	\[
		\sum_{j,l}\sum_{|\alpha|\leq m}\|\varphi_{j,l}\rho(x)^{|\alpha|-a}\partial^\alpha u|L_p(D)\|^p
			\sim\sum_{j,l}\sum_{|\alpha|\leq m}\|\rho(x)^{|\alpha|-a}\partial^\alpha (\varphi_{j,l}u)|L_p(D)\|^p
	\]
	for all $j$ and $l$. But this follows with analogous arguments as in the previous step, additionally using the
	condition $|\partial^\alpha\varphi_{j,l}(x)|\leq c_\alpha 2^{j|\alpha|}$ for the resolution of unity and the
	observation $\rho(x)\sim 2^{-j}$ for $x\in\supp\varphi_{j,l}$. In detail:
	
	Fix a multiindex $\alpha$  and put $J_{j,l}=\{(j',l'):\supp\varphi_{j',l'}\cap\supp\varphi_{j,l}\neq\emptyset\}$.
	Using Leibniz rule we find
	\begin{align*}
		\int_D\bigl|\rho(x)^{|\alpha|-a}\partial^\alpha (\varphi_{j,l}u)\bigr|^p dx
			&\lesssim\sum_{(j',l')\in J_{j,l}}
				\int_{2Q_{j,k_l}}\bigl|\varphi_{j',l'}(x)\rho(x)^{|\alpha|-a}\partial^\alpha (\varphi_{j,l}u)\bigr|^p dx\\
			&\lesssim\sum_{(j',l')\in J_{j,l}}2^{-j(|\alpha|-a)p}
				\int_{2Q_{j,k_l}}\biggl|\varphi_{j',l'}(x)\sum_{\beta\leq\alpha}
				\partial^{\alpha-\beta}\varphi_{j,l}(x)\partial^\beta u(x)\biggr|^p dx\\
			&\lesssim\sum_{(j',l')\in J_{j,l}}\sum_{\beta\leq\alpha}2^{-j(|\alpha|-a)p}2^{j(|\alpha|-|\beta|)p}
				\int_{2Q_{j,k_l}}\biggl|\varphi_{j',l'}(x)\partial^\beta u(x)\biggr|^p dx\\
			&\lesssim\sum_{(j',l')\in J_{j,l}}\sum_{\beta\leq\alpha}
				\int_{2Q_{j,k_l}}\bigl|\rho^{|\beta|-a}\varphi_{j',l'}(x)\partial^\beta u(x)\bigr|^p dx\\
			&\lesssim\sum_{(j',l')\in J_{j,l}}\sum_{|\beta|\leq m}
				\int_D\bigl|\rho^{|\beta|-a}\varphi_{j',l'}(x)\partial^\beta u(x)\bigr|^p dx\,.
	\end{align*}
	Summing over $\alpha$, $j$,  and $l$ gives the first half of the desired equivalence, using once more the controlled
	overlap of the supports of the functions $\varphi_{j,l}$ (the size of the sets $J_{j,l}$ is uniformly bounded, and
	vice versa, every pair $(j',l')$ belongs only to a uniformly bounded number of sets $J_{j,l}$). The other direction
	is easier,
	\begin{align*}
		\sum_{j,l}
			&\sum_{|\alpha|\leq m}\|\varphi_{j,l}\rho(x)^{|\alpha|-a}\partial^\alpha u|L_p(D)\|^p
				\sim\sum_{|\alpha|\leq m}\|\rho(x)^{|\alpha|-a}\partial^\alpha u|L_p(D)\|^p\\
			&\sim\sum_{|\alpha|\leq m}
					\int_D\biggl|\sum_{j,l}\rho^{|\alpha|-a}\partial^\alpha(\varphi_{j,l}u)\biggr|^p dx
				\lesssim\sum_{|\alpha|\leq m}\sum_{j,l}
					\int_D\bigl|\rho^{|\alpha|-a}\partial^\alpha(\varphi_{j,l}u)\bigr|^p dx\,.
	\end{align*}
	In the last step we used Step 1, i.e. the sum is locally finite (for every point $x$ there is only a uniformly
	bounded number of non-vanishing summands).
\end{proof}

\subsection{Invariance under diffeomorphisms}
\label{ssec-diffeo}

So far, all results in this section were valid for domains of the form $D=\R^d\setminus S$, where $S$ was an arbitrary closed Lebesgue null-set. As we have seen in the proof of Corollary \ref{cor-Kmm}, via localization, arguments for polytopes may be reduced to domains $D=\R^d\setminus\R^\ell$. We want to pursue this idea further. In particular, here we will investigate the behavior of the refined localization spaces under diffeomorphisms to extend our results to domains with piecewise smooth boundary which arise as (smooth) deformations of polytopes.

\begin{definition}
	Let $U_1,U_2 \subset \R^d$ be arbitrary domains. A smooth map
	$\varphi=(\varphi_1,\ldots,\varphi_d):U_1 \rightarrow U_2$ is called a bounded diffeomorphism if it is invertible
	with inverse $\varphi^{-1}=(\varphi_1^{-1},\ldots,\varphi_d^{-1}):U_2 \rightarrow U_1$ and if the derivatives
	$D^{\alpha}\varphi_i$ and $D^{\alpha}\varphi_i^{-1}$  are bounded on $U_1$ and $U_2$ for all
	$\alpha \in \N^{d}$ and $i=1,\ldots,d$, respectively.
\end{definition}

\begin{lemma}\label{lemma-diffeo}
	Let $D_1 \subset \R^d$ be an arbitrary domain and let $\varphi:\R^d\rightarrow\R^d$ be a function such that
	$\varphi|_{D_1}$ is a bounded diffeomorphism onto $D_2=\varphi(D_1)$ and $\varphi^{-1}(D_2)=D_1$. Further,
	let $s>\sigma_{p,q}$, $0<p<\infty$,  and $0<q \leq \infty$. Then there exists a positive constant $c>0$ such that
	\[
		\|f(\varphi(\cdot))|F^{s,\text{rloc}}_{p,q}(D_1)\|
			\leq c\cdot \|f|F^{s,\text{rloc}}_{p,q}(D_2)\|\,,
			\qquad f\in F^{s,\text{rloc}}_{p,q}(D_2)\,,
	\]
	with $c$ independent of $f$.
\end{lemma}

Note that it does not suffice to assume $\varphi:D_1\rightarrow D_2$ to be a (bounded) diffeomorphism: In the definition of the spaces $F^{s,\text{rloc}}_{p,q}(D)$ all functions are required to be (restrictions of) tempered distributions on $\R^d$, i.e.,  every  function on $D$ is required to have some extension to $\R^d$. For that reason $\varphi$ needs to be defined on $\R^d$, but only its restriction to $D_1$ needs to be bijective and smooth.

\begin{proof}
	To prove the lemma, we use the alternative characterization \eqref{eq-refined-weight}. To distinguish between the
	two domains $D_1$ and $D_2$, we now denote $\delta_D(x):=\min\left(\dist(x,\partial D),1\right)$ for a domain
	$D \subset \R^d$. Therefore, it suffices to prove that there exists a $c>0$ such that
	\begin{equation}\label{eq-diff1}
		\|f(\varphi(\cdot))|F^{s}_{p,q}(D_1)\|
			\leq c\,\|f|F^{s}_{p,q}(D_2)\|\,
	\end{equation}
	and
	\begin{equation}\label{eq-diff2}
		\|\delta_{D_1}(\cdot)^{-s}f(\varphi(\cdot))|L_p(D_1)\|
			\leq c\,\|\delta_{D_2}(\cdot)^{-s}f(\cdot)|L_p(D_2)\|
	\end{equation}
	with $c$ independent of $f$.
	
	From the corresponding diffeomorphism results for $F^{s}_{p,q}(\R^d)$, see \cite{Tr92} or
	\cite[Theorem 4.16]{Scharf3}, one can directly derive \eqref{eq-diff1} using the definition by restriction of
	$F^{s}_{p,q}(D_1)$ and $F^{s}_{p,q}(D_2)$, respectively.
	
	Hence, it is left to prove \eqref{eq-diff2}. From the fact that $\varphi(x) \in \partial D_2$ if,  and only if, 
	$x \in \partial D_1$ as well as from the Lipschitz continuity of $\varphi$ and $\varphi^{-1}$ it follows that
	\[
		\delta_{D_1}(\cdot) \sim  \delta_{D_2}(\varphi(\cdot)),
	\]
	i.e.,   there exist constants $a,b$ with $0<a<b$ such that
	$
		a \cdot \delta_{D_1}(x) \leq \delta_{D_2}(\varphi(x)) \leq b \cdot \delta_{D_1}(x) 
	$ 
	for all $x \in D_1$. But now, we can use the known results for diffeomorphisms for $L_p$-spaces, see e.g.\
	\cite[Lemma 4.13]{Scharf3}, to get
	\begin{align*}
		\|\delta_{D_1}(\cdot)^{-s}f(\varphi(\cdot))|L_p(D_1)\|
			&\leq C \cdot \|\delta_{D_2}(\varphi(\cdot))^{-s}f(\varphi(\cdot))|L_p(D_1)\| \\
			&\leq C' \cdot \|\delta_{D_2}(\cdot)^{-s}f(\cdot)|L_p(D_2)\|
	\end{align*}
	with $C'>0$ independent of $f$. Hence we have shown \eqref{eq-diff2} and the proof is finished.
\end{proof}

\begin{remark}
	As in \cite{Scharf3}we can relax the assumption that the diffeomorphism $\varphi$ admits bounded partial
	derivatives $D^\alpha\varphi$ for all $\alpha\in\N_0^d$. In fact, it suffices to assume $\varphi\in C^k(\R^d)$, with
	derivatives $D^\alpha\varphi$ bounded on $D_1$ for $|\alpha|\leq k$ for some $k>s$.
\end{remark}

\section{An embedding result}
\label{sec-embedding}

As demonstrated in \cite{Hansen1,Hansen2}, in connection with convergence rates for adaptive approximation schemes, embeddings of Kondratiev spaces into Besov or Triebel-Lizorkin spaces are of special interest. In \cite{Hansen2} necessary and sufficient conditions for embeddings
\[
	\calk^m_{a,p}(D)\cap B^s_{p,\infty}(D)\hookrightarrow F^m_{\tau,\infty}(D)
\]
were given. In this section we return to this subject once more, for two reasons: On the one hand we will demonstrate that the intersection on the left-hand side can be avoided, i.e. we will show embeddings for the Kondratiev spaces themselves, and on the other hand, we will show that Kondratiev spaces actually embed into the smaller spaces $F^{m,\text{rloc}}_{\tau,2}(D)$. We will prove these results using two different approaches: The first one is based on Theorem \ref{thm-Kmm}, Lemma \ref{lemma-Kma} and a variant of H\"older's inequality for Triebel-Lizorkin spaces, and the second one relies on the localization from Theorem \ref{thm-localize}.

\subsection{Approach via H\"older inequalities}

Given the results of Section \ref{ssec-Kmm}, in principle we can reduce embeddings between Kondratiev spaces $\calk^m_{a,p}(D)$ and the refined localization spaces $F^{m,\text{rloc}}_{\tau,q}(D)$ to estimates for
$\|\rho^{a-m}u|F^{m,\text{rloc}}_{\tau,q}(D)\|$ in terms of $\|u|F^{m,\text{rloc}}_{p,2}(D)\|$,
\begin{align*}
	\calk^m_{a,p}(D)\hookrightarrow F^{m,\text{rloc}}_{\tau,q}(D)
		&\iff\|u|F^{m,\text{rloc}}_{\tau,q}(D)\|\lesssim\|u|\calk^m_{a,p}(D)\|
			\sim\|\rho^{m-a}u|\calk^m_{m,p}(D)\|\\
		&\iff\|\rho^{a-m}u|F^{m,\text{rloc}}_{\tau,q}(D)\|\lesssim\|u|F^{m,\text{rloc}}_{p,2}(D)\|\,.
\end{align*}
In other words, we shall try to prove
\[
	\calk^m_{a,p}(D)=\rho^{a-m}\calk^m_{m,p}(D)
		=\rho^{a-m}F^{m,\text{rloc}}_{p,2}(D)
		\hookrightarrow F^{m,\text{rloc}}_{\tau,2}(D)\,.
\]
In the special case $p>\tau>1$ the last space again is $\calk^m_{m,\tau}(D)$. Hence,  in this special case our desired embedding result can be reformulated as
\[
	\calk^m_{a,p}(D)\hookrightarrow\calk^m_{m,\tau}(D)\,,\quad 1<\tau<p<\infty\,,
\]
which of course in this version can be proved more directly using the classical H\"older inequality. We are thus looking for an extension of that fact to $\tau\leq 1$.

In view of the identity \eqref{eq-refined-weight}, necessary conditions can be derived in an even easier way: By testing the inequality
\[
	\|\rho^{-m}u|L_\tau(D)\|\lesssim\|u|\calk^m_{a,p}(D)\|
\]
against ``typical'' representatives for these function spaces. In other words: If we can find a function $u\in\calk^m_{a,p}(D)$ such that $\|\rho^{-m}u|L_\tau(D)\|$ is infinite, the embedding is necessarily false.

\begin{prop}\label{prop-necessary}
Let  $m\in \N$,  $ 1<p<\infty$,     $0< \tau<p$,  and  $\delta$ denote  the dimension of the singular set of the domain $D$. 
	Moreover,     let either
	\[
		m-a=(d-\delta)\Bigl(\frac{1}{\tau}-\frac{1}{p}\Bigr)\quad\text{and}\quad a<m
	\]
	or
	\[
		m-a>(d-\delta)\Bigl(\frac{1}{\tau}-\frac{1}{p}\Bigr). 
	\]
	Then there exists a function $u\in\calk^m_{a,p}(D)$ such that $\|\rho^{-m}u|L_\tau(D)\|=\infty$. Consequently,
	\[
		\calk^m_{a,p}(D)\not\hookrightarrow F^{m,\text{rloc}}_{\tau,2}(D). 
	\]
	
\end{prop}

\begin{remark}\label{rem-embedding-technical}
Note that the spaces $F^{m,\text{rloc}}_{\tau,2}(D)$ are only defined for 
	  $m>\sigma_{\tau,2}=d \big( \frac{1}{\min(1,\tau)}-1\big)$.  We omitted this additional restriction in Proposition \ref{prop-necessary} since   for $m\leq \sigma_{\tau,2}$ the embedding then clearly does not exist.  We also refer to Remark \ref{rem-embedding-technical} in this context. 
\end{remark}

\begin{proof}
	{\bf Step 1:} We start with the condition $m-a=d(\frac{1}{\tau}-\frac{1}{p})$, and consider first the case $\delta=0$ and the
	domain $D=\R^d\setminus\{0\}$. We put
	\[
		u_\lambda=\rho^{m-d/\tau}(1+|\log\rho|)^\lambda
			=\rho^{a-d/p}(1+|\log\rho|)^\lambda\,.
	\]
	When calculating the $L_p$ and $L_\tau$-norms we shall only integrate over a (small) ball around $0$, which is
	clearly equivalent to multiplying $u_\lambda$ with a cut-off function. We then find
	\[
		\|\rho^{-m}u_\lambda|L_\tau(D)\|^\tau
			\sim\int_0^R\Bigl(\rho^{-m}\rho^{m-d/\tau}(1+|\log\rho|)^\lambda\Bigr)^{\tau}\rho^{d-1}d\rho
			=\int_0^R \rho^{-1}(1+|\log\rho|)^{\lambda\tau}d\rho\,,
	\]
	which is finite if, and only if,  $\lambda\tau<-1$.
	
	To calculate the norm in $\calk^m_{a,p}(D)$, we first note
	\[
		|\partial^\alpha u_\lambda|
			\lesssim\rho^{a-d/p-|\alpha|}\sum_{j=0}^{|\alpha|}(1+|\log\rho|)^{\lambda-j}\,.
	\]
	Then we find
	\[
		\|\rho^{|\alpha|-a}\partial^\alpha u_\lambda|L_p(D)\|^p
			\lesssim\sum_{j=0}^{|\alpha|}\int_0^R \rho^{-1}(1+|\log\rho|)^{(\lambda-j)p}d\rho\,,
	\]
	which is finite if, and only if,  $\lambda p<-1$.
	
	Thus whenever $\tau<p$ we can find a parameter $\lambda$ such that $-\frac{1}{\tau}<\lambda<-\frac{1}{p}$,
	which yields a function with the prescribed properties.
	
	\vspace{3mm}
	
	{\bf Step 2:} For the case $\delta>0$ clearly $D=\R^d\setminus\R^\delta$ is prototypical. However, this can
	immediately be traced back to the case $\R^{d-\delta}\setminus\{0\}$, by simply identifying the corresponding
	example in $d-\delta$ variables with a function in $d$ variables which is constant w.r.t. the last $\delta$ variables.
	
	For general polytopes we simply restrict ourselves to sufficiently small neighborhoods of parts of the singular set
	(single vertices, edges etc.). Then, after rotation and translation, the neighborhood coincides with the situation for
	$\R^d\setminus\R^\ell$.
	
	\vspace{3mm}
	
	{\bf Step 3:} The case $m-a	>(d-\delta)(\frac{1}{\tau}-\frac{1}{p})$ follows by basically the same arguments,
	considering the function $u=\rho^{m-d/\tau}$.
\end{proof}

\begin{remark}
	The resulting necessary conditions (and the test functions in the proof) are exactly the same ones as previously
	obtained in \cite{Hansen2} w.r.t. the space $F^m_{\tau,q}(D)$, but (given \eqref{eq-refined-weight}) found by
	more elementary means.
\end{remark}

There is another necessary condition which is well-known for Lebesgue spaces on finite measure spaces.

\begin{lemma}\label{lemma-necessary}
	A continuous embedding $\calk^m_{a,p}(D)\hookrightarrow F^{m,\text{rloc}}_{\tau,q}(D)$ implies $\tau\leq p$.
\end{lemma}

\begin{proof}
	This is based on the observation that on any bounded subdomain $D_0\subset D$ with positive distance to the
	singular set the weight function $\rho$ is bounded from below, hence for any function with compact support in
	$D_0$ its norm in $\calk^m_{a,p}(D)$ is equivalent to its unweighted counterpart in $W^m_p(D_0)$. However,
	when comparing $W^m_p(D_0)$ and $F^m_{\tau,q}(D)$ the restriction $\tau\leq p$ is well-known. This can be
	made explicit by considering functions $|x-x_0|^\alpha$ with appropriate exponents $\alpha$ for arbitrary interior
	points $x_0\in D_0$. See also Proposition \ref{prop-radial} below.
\end{proof}

For the embedding itself, as announced above,  we shall estimate $\|\rho^{a-m}u|F^{m,\text{rloc}}_{\tau,q}(D)\|$ in terms of $\|u|F^{m,\text{rloc}}_{p,2}(D)\|$. Here and in what follows, we will mainly concentrate on the special case $D=\R^d\setminus\R^\ell$. This yields two important simplifications. Firstly, all function spaces on domains $D$ can be identified with the corresponding spaces on $\R^d$ (also for Triebel-Lizorkin spaces the restriction $s>\sigma_{p,q}$ guarantees that the spaces indeed consist of functions). Secondly, the distance function $\delta(x)=\dist(x,\partial D)$ itself is already smooth on $D$, i.e. we simply have $\rho=\delta$.

In view of \eqref{eq-refined-weight} it is therefore sufficient to show the estimates
\[
	\|\rho^{a-2m}u|L_\tau(\R^d)\|\lesssim\|\rho^{-m}u|L_p(\R^d)\|
		\qquad\text{and}\qquad
	\|\rho^{a-m}u|F^m_{\tau,q}(\R^d)\|\lesssim\|u|F^m_{p,2}(\R^d)\|\,.
\]
For the first part, applying H\"older's inequality is quite natural, and also for the second one we shall use a variant of H\"older's inequality for Triebel-Lizorkin spaces shown in \cite[Theorem 4.2.1]{SickelTriebel}. 

\begin{prop}\label{prop-Holder-F}
	Let $s>0$, $0<p_1,p_2,p<\infty$,  and $0<q_1,q_2,q\leq\infty$.  Moreover, assume 
	\[
		\frac{1}{r_1}=\frac{1}{p_1}-\frac{s}{d}>0\,,\quad
		\frac{1}{r_2}=\frac{1}{p_2}-\frac{s}{d}>0\,,\quad
		\frac{1}{r_1}+\frac{1}{r_2}=\frac{1}{r}=\frac{1}{p}-\frac{s}{d}<1\,.
	\]
	Then it  holds
	\begin{equation}\label{eq-Holder-F}
		F^s_{p_1,q_1}F^s_{p_2,q_2}\subset F^s_{p,q}
	\end{equation}
	if, and only if,  $\max(q_1,q_2)\leq q$. Therein \eqref{eq-Holder-F} is a short-hand notation for the pointwise
	multiplication inequality
	\[
		\|f\cdot g|F^s_{p,q}(\R^d)\|
			\lesssim\|f|F^s_{p_1,q_1}(\R^d)\|\cdot\|g|F^s_{p_2,q_2}(\R^d)\|
	\]
	for all $f\in F^s_{p_1,q_1}(\R^d)$ and $g\in F^s_{p_2,q_2}(\R^d)$.
\end{prop}

As a further preparation we cite the following result from \cite{RunstSickel}.

\begin{prop}\label{prop-radial}
	The radial function $f_\beta(x)=|x|^\beta\zeta(x)$, where $\zeta$ is a smooth radial cut-off function (i.e.
	$\zeta(x)=1$ for $|x|\leq 1$ and $\zeta(x)=0$ for $|x|\geq 2$) belongs to $F^s_{p,q}(\R^d)$ if, and only if,  
	$s<\frac{d}{p}+\beta$.
\end{prop}

\begin{remark}\label{remark-radial}
	Though formulated for radial functions, this clearly can be transferred to functions $\rho(x)^\beta$ and the domains
	$\R^d\setminus\R^\ell$ as in the proof of Proposition \ref{prop-necessary}. The resulting condition then reads as
	$s<\frac{d-\ell}{p}+\beta$ (independent of $q$).
\end{remark}

\begin{thm}\label{thm-Holder}
	Let $m\in\N$, $a\in\R$, $1<p<\infty$, $0<\tau<p$,   $0<q_2\leq q_1\leq\infty$, and $m>\sigma_{\tau,2}$.  Let
	$D=\R^d\setminus\R^\ell$ with $0\leq\ell\leq d-2$  and assume
	\[
		\frac{d+\ell}{d}m-a<(d-\ell)\Bigl(\frac{1}{\tau}-\frac{1}{p}\Bigr)
			\quad\text{and}\quad\frac{1}{\tau}-1<\frac{m}{d}<\frac{1}{p}\,.
	\]
	Then we have an estimate
	\begin{equation}\label{eq-Holder-rloc}
		\|\rho^{a-m}u|F^{m,\text{rloc}}_{\tau,q_1}(D)\|
			\lesssim\|\rho^{a-m}\zeta|F^m_{\eta,q_2}(\R^d)\|\cdot\|u|F^{m,\text{rloc}}_{p,q_2}(D)\|
	\end{equation}
	for all $u\in F^{m,\text{rloc}}_{p,q_2}(D)$ with compact support. Therein $\zeta$ is a smooth cut-off function
	w.r.t. $\supp u$ (i.e. compactly supported, infinitely differentiable, and $\zeta(x)=1$ on $\supp u$), and
	$\frac{1}{\eta}=\frac{m}{d}+\frac{1}{\tau}-\frac{1}{p}$. \\
	Moreover,  the following embedding  holds 
	\[
	\calk^m_{a,p}(D)\hookrightarrow F^{m,\text{rloc}}_{\tau,q_1}(D). 
	\]
\end{thm}

\begin{proof}
	Recall that in order to establish the embedding  it suffices to show the two estimates
	\[
		\|\rho^{a-2m}u|L_\tau(\R^d)\|\lesssim\|\rho^{-m}u|L_p(\R^d)\|
		\qquad\text{and}\qquad
		\|\rho^{a-m}u|F^m_{\tau,q_1}(\R^d)\|\lesssim\|u|F^m_{p,q_2}(\R^d)\|\,.
	\]
	
	{\bf Step 1:} We first show the estimate
	\begin{equation}\label{est-1}
		\|\rho^{a-2m}u|L_\tau(\R^d)\|\lesssim  \|\rho^{a-m}\zeta| L_r(\R^d)\|\cdot \|\rho^{-m}u|L_p(\R^d)\|
	\end{equation}
	with $\frac{1}{r}=\frac{1}{\tau}-\frac{1}{p}$. In view of the Sobolev-type embedding
	$F^m_{\eta,q_2}(\R^d)\hookrightarrow L_r(\R^d)$ and \eqref{eq-refined-weight}, this implies the first half of the estimate
	\eqref{eq-Holder-rloc}.
	
	Since by assumption $p>\tau$, applying the classical H\"older inequality with
	$1=\frac{\tau}{p}+\frac{p-\tau}{p}$ yields
	\begin{equation}\label{est-01}
		\int_D\bigl|\rho(x)^{a-2m}u(x)\bigr|^\tau dx
			\leq\biggl(\int_{B_{R,d}}\rho^{(a-m)\frac{\tau p}{p-\tau}}\zeta(x)^{\frac{\tau p}{p-\tau}}dx
				\biggr)^{\frac{p-\tau}{p}}\biggl(\int_D\rho^{-mp}|u(x)|^p dx\biggr)^{\frac{\tau}{p}}\,,
	\end{equation}
	where $B_{R,d}$ stands for the ball in $\R^d$ around $0$ with radius $R$,  which gives  \eqref{est-1}.   
We additionally  show that this implies 	\begin{equation}\label{est-2}
		\|\rho^{a-2m}u|L_\tau(\R^d)\|\lesssim\|\rho^{-m}u|L_p(\R^d)\| 
	\end{equation}
	by proving that the first factor in \eqref{est-01}  is finite. 
	Since as a smooth function $\zeta$ is
	bounded, it will be omitted below.  
	Using the special structure of
	$D=\R^d\setminus\R^\ell \equiv \R^d\setminus (\R^\ell \times\{0\}^{d-\ell})$, it follows from splitting $x=(x',x'')$ with $x'\in\R^\ell$ and
	$x''\in\R^{d-\ell}$ that $\rho(x)=|x''|$; hence we can rewrite the integral using polar coordinates in $\R^{d-\ell}$
	\[
		\int_{B_{R,d}}\rho^{(a-m)\frac{\tau p}{p-\tau}}dx
			=\int_{B_{R,d}}|x''|^{(a-m)\frac{\tau p}{p-\tau}}dx
			\lesssim R^\ell\int_0^R \varrho^{(a-m)\frac{\tau p}{p-\tau}}\varrho^{d-\ell-1}d\varrho\,,
	\]
	observe $B_{R,d}\subset [-R,R]^\ell\times B_{R,d-\ell}$. The last integral is finite if, and only if
	\[
		(a-m)\frac{\tau p}{p-\tau}+d-\ell>0\iff m-a<(d-\ell)\Bigl(\frac{1}{\tau}-\frac{1}{p}\Bigr)\,.
	\]
	
	{\bf Step 2:} For the second half of the estimate \eqref{eq-Holder-rloc},  
	we shall apply Proposition \ref{prop-Holder-F}. For this, we put
	$\frac{1}{\eta}=\frac{m}{d}+\frac{1}{\tau}-\frac{1}{p}$. Since $\tau<p$ we have $0<\eta<\infty$ and also
	$\frac{1}{\eta}-\frac{m}{d}>0$. Together with the assumption $\frac{1}{p}>\frac{m}{d}> \frac{1}{\tau}-1$, all
	conditions for Proposition \ref{prop-Holder-F} are satisfied.  Then the inequality corresponding  to \eqref{eq-Holder-F}   with our choice of parameters now reads as
	\begin{equation}\label{est-3}
		\|\rho^{a-m}\cdot u|F^m_{\tau,q_1}(\R^d)\|
			\lesssim\|u|F^m_{p,q_2}(\R^d)\|\cdot\|\rho^{a-m}\zeta|F^m_{\eta,q_2}(\R^d)\|
	\end{equation}
	which together with  \eqref{eq-refined-weight} and Step 1 shows   \eqref{eq-Holder-rloc}.  Moreover,  from \eqref{est-3} we obtain 
	\[
		\|\rho^{a-m}u|F^m_{\tau,q_1}(\R^d)\|\lesssim\|u|F^m_{p,q_2}(\R^d)\|
	\]
	if we can verify $\rho^{a-m}\zeta\in F^m_{\eta,q_2}(\R^d)$. But this follows from Proposition \ref{prop-radial}
	and Remark \ref{remark-radial}: The resulting condition $m<\frac{d-\ell}{\eta}+a-m$ is satisfied in view of
	\[
		m<\frac{d-\ell}{\eta}+a-m
			\iff m-a<(d-\ell)\Bigl(\frac{m}{d}+\frac{1}{\tau}-\frac{1}{p}\Bigr)-m
				=(d-\ell)\Bigl(\frac{1}{\tau}-\frac{1}{p}\Bigr)-\frac{\ell}{d}m\,.
	\]
	This  shows the embedding and completes the proof. 
\end{proof}

\begin{remark} 
In conclusion,    we see from Theorem \ref{thm-Holder} 
that  the approach via H\"older inequalities sketched at the beginning of the section indeed turns out to be a viable option in order to establish embeddings between Kondratiev and refined localization spaces.  However, while the advantage of these H\"older-type arguments lies in their simplicity (they are the straightforward generalization of the arguments one would use for $\calk^m_{a,p}(D)\hookrightarrow\calk^m_{m,\tau}(D)$), so far they do not provide sufficient conditions matching the ones from Proposition \ref{prop-necessary} and Lemma \ref{lemma-necessary}. In the arguments above we have picked  up the quite restrictive additional condition
$\frac{1}{\tau}-1<\frac{m}{d}<\frac{1}{p}$, particularly $\frac{1}{\tau}-\frac{1}{p}<1$, and also the other condition is suboptimal because of the factor $\frac{d+\ell}{d}$ (except for the case $\ell=0$). 
This is  due to the fact that   the current version of Proposition \ref{prop-Holder-F} is formulated for arbitrary functions
$f\in F^s_{p_1,q_1}(\R^d)$ and therefore its application to $f=\rho^{a-m}$ (naturally) leads to suboptimal conditions on the parameters involved. Particularly the property that $\rho^{a-m}$  is essentially constant along $\R^\ell$ is neglected.  
In \cite[Appendix D]{hansen-habil}  possible modifications/improvements of this approach are presented, removing the mentioned factor $\frac{d+\ell}{d}$. This is done by 
invoking another property of Triebel-Lizorkin spaces, the Fubini-property (see \cite[Theorem 4.4]{Tr01}, \cite[Proposition 4]{Scharf}, \cite[Proposition 1.11]{Scharf2}) and the resulting conditions for the embedding in  Theorem  \ref{thm-Holder}  read as 
\[
		m-a<(d-\ell)\Bigl(\frac{1}{\tau}-\frac{1}{p}\Bigr)
			\quad\text{and}\quad\frac{1}{\tau}-1<\frac{m}{d-\ell}<\frac{1}{p}\,.
	\]
On the other hand,  we see that those arguments require the even more restrictive condition $\frac{1}{\tau}-1<\frac{m}{d-\ell}<\frac{1}{p}$. 
For that reason they are not repeated here.  In the next section we shall present another approach which remedies these shortcomings. 
\end{remark}

\subsection{Second approach via localization}
\label{ssec-embedding-localize}

We start with the prototypical situation $D=\R^d\setminus\R^\ell$, $0\leq\ell <d$. 

\begin{thm}\label{thm-embedding}
	Let $m\in\N$,  $1<p<\infty$,   $0<\tau<\infty$,  where $m>\sigma_{\tau,2}=d \big( \frac{1}{\min(1,\tau)}-1\big)$,  and $a\in\R$.  Moreover, assume either
	\[
		\tau<p\quad\text{and}\quad m-a<(d-\ell)\Bigl(\frac{1}{\tau}-\frac{1}{p}\Bigr)
	\]
	or $\tau=p$ and $a\geq m$. Furthermore, let $u\in\calk^m_{a,p}(\R^d\setminus\R^\ell)$ have compact support.
	Then it holds
	\[
		\|u|F^{m,\text{rloc}}_{\tau,2}(\R^d\setminus\R^\ell)\|
			\lesssim\|u|\calk^m_{a,p}(\R^d\setminus\R^\ell)\|\,.
	\]
\end{thm}

\begin{remark}\label{rem-embedding-technical}
	The technical condition $m>d \big( \frac{1}{\min(1,\tau)}-1\big)$ stems from the definition of the space
	$F^{m,\text{rloc}}_{\tau,2}(D)$ and is only an additional  restriction if  $\tau< 1$.  It  results in the condition
	$$
		(d-\ell)\Bigl(\frac{1}{\tau}-\frac{1}{p}\Bigr)+a>d\Bigl(\frac{1}{\tau}-1\Bigr)
			\iff a>\ell\Bigl(\frac{1}{\tau}-\frac{1}{p}\Bigr)+d\Bigl(\frac{1}{p}-1\Bigr)\,.
	$$
	However, it should be noted that $m>\sigma_{\tau,2}$ automatically is fulfilled whenever we have an
	embedding $F^m_{\tau,2}(\R^d)\hookrightarrow L_p(\R^d)$, since such an embedding requires $\tau\leq p$
	and	$m>d(\frac{1}{\tau}-\frac{1}{p})\geq\sigma_{\tau,2}$. 
	We   conjecture that  
	there is a counterpart of Theorem \ref{thm-embedding} also for
	$m\leq\sigma_{\tau,2}$, 
	but since this requires extending the definition of refined localization
	spaces as well as the characterization \eqref{eq-refined-weight} to   parameters $m\leq\sigma_{\tau,2}$,  this will be
	postponed to future publications.
\end{remark}

\begin{proof}
	Let a Whitney decomposition of $\R^d\setminus\R^\ell$ be given, together with a corresponding resolution of
	unity as in Definition \ref{def-rloc}. As a consequence of the assumed compact support of $u$, all numbers $N_j$
	can be treated as being finite: We can restrict our considerations to a decomposition of
	$[-2^J,2^J]^d\setminus\R^\ell$ for some sufficiently large integer $J$, where the condition
	\eqref{whitney-decomp} is only enforced w.r.t. the distance to $\R^\ell$. The main aspect of this reduction is that
	now all numbers $N_j$ are finite, together with an estimate $N_j\leq c_0 2^{j\ell}$ for $j\geq J$.
	
	Our starting point is the corresponding decomposition of Kondratiev spaces from Theorem \ref{thm-localize},
	\[
		\|u|\calk^m_{a,p}(\R^d\setminus\R^\ell)\|^p
			\sim\sum_{j=0}^\infty\sum_{l=1}^{N_j}\|\varphi_{j,l}u|\calk^m_{a,p}(\R^d\setminus\R^\ell)\|^p\,.
	\]
	
	We then note that for $j\leq j_0$ we have $\rho(x)\sim 1$ for all $x\in 2Q_{j,k_l}$, and for $j\geq j_0$ we have
	$\rho(x)\sim 2^{-j}$ on $2Q_{j,k_l}$. Hence we can reformulate the localized norm
	\[
		\|u|\calk^m_{a,p}(\R^d\setminus\R^\ell)\|^p
			\sim\sum_{j<j_0}\sum_{l=1}^{N_j}\|\varphi_{j,l}u|W^m_p(2Q_{j,k_l})\|^p
				+\sum_{j\geq j_0}\sum_{l=1}^{N_j}\|\varphi_{j,l}u|\calk^m_{a,p}(\R^d\setminus\R^\ell)\|^p\,.
	\]
	For the first sum, we use the identification $W^m_p(\Omega)=F^m_{p,2}(\Omega)$, together with the standard
	embedding $L_p(\Omega)\hookrightarrow L_\tau(\Omega)$, to obtain
	\[
		\|\varphi_{j,l}u|W^m_p(2Q_{j,k_l})\|
			\gtrsim\|\varphi_{j,l}u|F^m_{\tau,2}(2Q_{j,k_l})\|=\|\varphi_{j,l}u|F^m_{\tau,2}(\R^d)\|\,.
	\]
	The second type of terms we have to treat with more care. Here we shall apply a scaling argument: A
	substitution $y=2^{j-j_0}x$ maps $Q_{j,k_l}$ onto $Q_{j_0,k_l}$, thus,  
	\begin{align*}
		\|&\varphi_{j,k_l}u|\calk^m_{a,p}(\R^d\setminus\R^\ell)\|^p\\
			&=\sum_{|\alpha|\leq m}\int_{2Q_{j,k_l}}
				\bigl|\rho^{|\alpha|-a}\partial^\alpha(\varphi_{j,k_l}u)\bigr|^p dx\\
			&=\sum_{|\alpha|\leq m}\int_{2Q_{j_0,k_l}}\bigl|\rho^{|\alpha|-a}(2^{-j+j_0}y)
				\partial^\alpha\bigl(\varphi_{j,k_l}u\bigr)(2^{-j+j_0}\cdot)\bigr|^p 2^{d(-j+j_0)}dy\\
			&\sim 2^{d(-j+j_0)}\sum_{|\alpha|\leq m}2^{(|\alpha|-a)(-j+j_0)p}2^{|\alpha|(j-j_0)p}
				\int_{2Q_{j_0,k_l}}\bigl|\rho^{|\alpha|-a}(y)
				\partial^\alpha\bigl(\varphi_{j,k_l}(2^{-j+j_0}\cdot)u(2^{-j+j_0}\cdot)\bigr)(y)\bigr|^p dx\\
			&\sim 2^{(d-ap)(-j+j_0)}
				\|\varphi_{j,k_l}(2^{-j+j_0}\cdot)u(2^{-j+j_0}\cdot)|W^m_p(2Q_{j_0,k_l})\|^p\,.
	\end{align*}
	This needs once more to be combined with $W^m_p(\Omega)=F^m_{p,2}(\Omega)$ and
	$L_p(\Omega)\hookrightarrow L_\tau(\Omega)$ for $\Omega=Q_{j_0,k_l}$ (we always have $\tau\leq p$).
	Moreover, we need also the scaling properties of Triebel-Lizorkin spaces \cite{SchnVyb}, which gives
	\[
		\|\varphi_{j,k_l}(2^{-j+j_0}\cdot)u(2^{-j+j_0}\cdot)|F^m_{\tau,2}(\R^d)\|
			\sim 2^{(-j+j_0)(m-d/\tau)}\|\varphi_{j,k_l}u|F^m_{\tau,2}(\R^d)\|\,.
	\]
	Altogether we have found
	\begin{align*}
		\|u|&\calk^m_{a,p}(\R^d\setminus\R^\ell)\|^p\\
			&\gtrsim\sum_{j<j_0}\sum_{l=1}^{N_j}\|\varphi_{j,l}u|F^m_{\tau,2}(\R^d)\|^p
				+\sum_{j\geq j_0}\sum_{l=1}^{N_j}2^{(m-a-d(\frac{1}{\tau}-\frac{1}{p}))(-j+j_0)p}
					\|\varphi_{j,k_l}u|F^m_{\tau,2}(\R^d)\|^p\,.
	\end{align*}
	From this, the case $\tau=p$ and $m-a\leq 0$ immediately follows. In the case $\tau<p$ it remains to switch from
	$p$-summation to $\tau$-summation. Here also the estimate for $N_j$ comes into play. We can apply H\"older's
	inequality twice, and obtain with $1=\frac{\tau}{p}+\frac{p-\tau}{p}$ and
	$\gamma=m-a-(d-\ell)(\frac{1}{\tau}-\frac{1}{p})<0$
	\begin{align*}
		\|u|&\calk^m_{a,p}(\R^d\setminus\R^\ell)\|^\tau\\
			&\gtrsim\sum_{j<j_0}\sum_{l=1}^{N_j}\|\varphi_{j,l}u|F^m_{\tau,2}(\R^d)\|^\tau\\
			&\qquad +\Biggl(\sum_{j\geq j_0}2^{\gamma(-j+j_0)p}2^{jl\frac{p-\tau}{\tau}}
				\biggl(\sum_{l=1}^{N_j}\|\varphi_{j,k_l}u|F^m_{\tau,2}(\R^d)\|^p
				\biggr)^{\frac{\tau}{p}\cdot\frac{p}{\tau}}\Biggr)^{\frac{\tau}{p}}
				\biggl(\sum_{j\geq j_0}2^{\gamma(j-j_0)\frac{p\tau}{p-\tau}}\biggr)^{\frac{p-\tau}{p}}\\
			&\geq\sum_{j<j_0}\sum_{l=1}^{N_j}\|\varphi_{j,l}u|F^m_{\tau,2}(\R^d)\|^\tau
				+\sum_{j\geq j_0}N_j^{\frac{p-\tau}{p}}
					\biggl(\sum_{l=1}^{N_j}\|\varphi_{j,k_l}u|F^m_{\tau,2}(\R^d)\|^p\biggr)^{\frac{\tau}{p}}\\
			&\geq\sum_{j<j_0}\sum_{l=1}^{N_j}\|\varphi_{j,l}u|F^m_{\tau,2}(\R^d)\|^\tau
				+\sum_{j\geq j_0}\sum_{l=1}^{N_j}\|\varphi_{j,k_l}u|F^m_{\tau,2}(\R^d)\|^\tau
				\sim\|u|F^{m,\text{rloc}}_{\tau,2}(\R^d\setminus\R^\ell)\|^\tau\,.
	\end{align*}
	This completes the proof.
\end{proof}

As an immediate consequence, we find:

\begin{thm}\label{thm-embedding-2}
	Let $D\subset\R^d$ be a bounded Lipschitz domain with piecewise smooth boundary,  let   $\delta$ denote  the dimension of the singular set of the domain $D$,  and 
let $m\in \N$, $a\in \R$,  $1<p<\infty$, and $0<\tau<\infty$,  where $m>d\big(\frac{1}{\min(1,\tau)}-1\big)$.  Then the embedding
\[
\calk^m_{a,p}(D)\hookrightarrow F^m_{\tau,2}(D)
\]
holds if,  and only if,  either $\tau<p$ and $m-a<(d-\delta)(\frac{1}{\tau}-\frac{1}{p})$ or $\tau=p$ and $m\leq a$.
\end{thm}

\begin{proof}
	This follows by the same localization arguments as in the proof of Corollary \ref{cor-Kmm} together with
	the invariance under diffeomorphisms (Lemma \ref{lemma-diffeo})and the obtained necessary conditions in Proposition \ref{prop-necessary}.
\end{proof}

\begin{remark}
	(i) On polytopes the condition will read as $m-a<2(\frac{1}{\tau}-\frac{1}{p})$, independent of the dimension,
	since the singular set will always contain $(d-2)$-faces.
	
	
	(ii) If we compare these embeddings with the standard one $W^m_p(D)\hookrightarrow F^m_{\tau,2}(D)$, which
	in turn is directly based on $L_p(D)\hookrightarrow L_\tau(D)$, we can also use the following interpretation: This
	standard embedding extends to functions with controlled/limited blow-up for their higher derivatives (hence an
	upper bound on the exponent $m-a$ of the weight function for $m$th order derivatives), at the expense of more
	rigid behavior of low order derivatives (here a certain decay towards the singular set is required, which for
	increasing $m$ then translates to vanishing traces).
\end{remark}

To close this subsection, we shall combine the embedding with the regularity result from Proposition
\ref{prop-regularity}.

\begin{thm}\label{thm-reg-pde}
	Let $D\subset\R^d$, $d=2,3$, be a bounded polyhedral domain without cracks. Further, let the matrix function
	$A:D\rightarrow\bigl(\mathcal{K}^m_{0,\infty}(D)\bigr)^{d\times d}$ fulfill the assumptions from Proposition
	\ref{prop-regularity}. Then, for all $f\in\calk^{m-1}_{a-1}(D)$, $|a|<\min(\overline a,m)$, the solution $u$ to
	\eqref{eq:PDE} belongs to the space $F^{m+1}_{\tau,2}(D)$, where
	$\tau<\bigl(\frac{m-a}{d-\delta}+\frac{1}{2}\bigr)^{-1}$.
\end{thm}

\subsection{The reverse embedding}

Though of less interest for applications such as adaptive approximation of solution of elliptic PDEs, we shall briefly discuss the reverse embeddings $F^{m,\text{rloc}}_{p,2}(\R^d\setminus S)\hookrightarrow\calk^m_{a,p}(\R^d\setminus S)$, as they complement the results of the previous section and further demonstrate the close relations between Kondratiev and refined localization spaces.

\begin{prop}\label{prop-necessary-2}
	  Let    $\delta$ denote  the dimension of the singular set of the domain $D$ and let either
	\[
		m-a=(d-\delta)\Bigl(\frac{1}{\tau}-\frac{1}{p}\Bigr)\quad\text{and}\quad a>m
	\]
	or
	\[
		m-a<(d-\delta)\Bigl(\frac{1}{\tau}-\frac{1}{p}\Bigr)\,.
	\]
	Then there exists a function $u\in F^{m,\text{rloc}}_{\tau,2}(D)\setminus\calk^m_{a,p}(D)$. Moreover, a
	continuous embedding $F^{m,\text{rloc}}_{\tau,2}(D)\hookrightarrow\calk^m_{a,p}(D)$ implies $p\leq\tau$.
\end{prop}

For the proof we intend to use essentially the same test functions as in Proposition \ref{prop-necessary}, but now we need to additionally consider their membership in $F^m_{\tau,2}(D)$. Apart from Proposition \ref{prop-radial}, we will also need the following refined version, which as well  can be found in \cite[Sect.~3.3.1,  Lemma 1]{RunstSickel}.

\begin{lemma}\label{lemma-radial-log}
	The radial function $f_{\beta,\gamma}(x)=|x|^\beta(1+|\log x|)^\gamma\zeta(x)$, where $\zeta$ is a smooth radial
	cut-off function, $\beta,\gamma\in\R$ with $\gamma>0$, belongs to $F^s_{p,q}(\R^d)$ if, and only if, 
	either $s<\frac{d}{p}+\beta$ or $s=\frac{d}{p}+\beta$ and $\gamma p<-1$.
\end{lemma}

\begin{proof}[\normalfont\bf Proof of Proposition \ref{prop-necessary-2}.]
	In the case $\delta=0$ and $m-a=d(\frac{1}{\tau}-\frac{1}{p})$ we consider
	$u_\lambda=\rho^{m-d/\tau}(1+|\log\rho|)^\lambda=\rho^{a-d/p}(1+|\log\rho|)^\lambda$, for which we found
	$u_\lambda\in\calk^m_{a,p}(D)$ if, and only if,  $\lambda p<-1$, as well as $\rho^{-m}u_\lambda\in L_\tau(D)$ if,
	and only if,  $\lambda\tau<-1$ (see Proposition \ref{prop-necessary}). On the other hand, by the lemma above we find that also
	$u_\lambda\in F^m_{\tau,2}(D)$ if, and only if,  $\lambda\tau<-1$ (note that $s=\frac{d}{p}+\beta$ corresponds to
	$m=\frac{d}{\tau}+m-\frac{d}{\tau}$). Choosing $-\frac{1}{p}<\lambda<-\frac{1}{\tau}$ gives the first part
	(obviously $a>m$ is equivalent to $\tau>p$).
	
	The second part similarly follows by considering $u=\rho^{a-d/p}$.
\end{proof}

\begin{thm}\label{thm-embedding-3}
	Let $m\in\N$, $a\in\R$, $1<p<\infty$,   $0<\tau<\infty$, and $0\leq l<d$.  Moreover, assume either
	\[
		\tau>p\quad\text{and}\quad m-a>(d-\ell)\Bigl(\frac{1}{\tau}-\frac{1}{p}\Bigr)
	\]
	or $\tau=p$ and $a\leq m$. Furthermore, let $u\in F^{m,\text{rloc}}_{\tau,2}(\R^d\setminus\R^\ell)$ have
	compact support. Then it holds
	\[
		\|u|\calk^m_{a,p}(\R^d\setminus\R^\ell)\|
			\lesssim\|u|F^{m,\text{rloc}}_{\tau,2}(\R^d\setminus\R^\ell)\|\,.
	\]
\end{thm}

\begin{proof}
	The result follows by exactly the same arguments as in the proof of Theorem \ref{thm-embedding}, simply by
	reversing the roles of $\calk^m_{a,p}$ and $F^{m,\text{rloc}}_{\tau,2}$, i.e.,  starting with the definition of
	$F^{m,\text{rloc}}_{\tau,2}$, scaling of the terms $\varphi_{j,l}u$, using $L_\tau\hookrightarrow L_p$,
	re-scaling and changing to $p$-summation towards the localization of Kondratiev spaces.
\end{proof}

\end{document}